\documentclass{article}
\usepackage{latexsym}
\usepackage{amssymb}
\usepackage{amsmath}
\usepackage{color}
\usepackage{graphicx}
\usepackage{amsthm}
\usepackage{indentfirst}
\usepackage{mathrsfs}
\usepackage{dsfont}
\usepackage{enumitem}
\usepackage{multirow}
\usepackage{extpfeil,extarrows}
\usepackage{galois}
\usepackage{float}
\usepackage{longtable}
\usepackage{pdfpages}
\usepackage{appendix}
\allowdisplaybreaks

\newtheorem{theorem}{\color{black}\indent Theorem}[section]

\newtheorem{lemma}{\color{black}\indent Lemma}[section]
\newtheorem{remark}{\color{black}\indent Remark}

\newtheorem{cor}{\color{black}\indent Corollary}[section]
\newtheorem{pro}{\color{black}\indent Proposition}[section]

\begin{document}

\title{The multi-scale KAM persistence without a scaling order for Hamiltonian systems}
\author{ Weichao Qian$^a$\thanks{E-mail address: qian\_wc@163.com}, ~Yong Li$^{a,b}$\thanks{E-mail address: liyong@jlu.edu.cn},~ Xue Yang$^{a,b}$\thanks{E-mail address: xueyang@jlu.edu.cn}~\footnote{Corresponding author},
\\
{$^a$College of Mathematics, Jilin University,}
\\
{ Changchun 130012, P. R. China.}
\\
{$^b$School of Mathematics and statistics and}
\\
{Center for Mathematics and Interdisciplinary Sciences, }
\\
{Northeast Normal University, Changchun 130024, P. R. China}
}
\date{}

\maketitle

\begin{abstract}
The persistence of invariant tori in multi-scale Hamiltonian systems is intrinsically linked to the stability of the N-body problem. However, the existing non-degeneracy conditions in disordered scenarios have been formulated too generally, making them difficult to apply directly to celestial mechanics. In this work, we present a readily verifiable non-degeneracy condition for the persistence of invariant tori in disordered multi-scale Hamiltonian systems.

\end{abstract}
\section{Background}
In the N-body problem, the masses of celestial bodies and the distances between them are crucial for characterizing the system. When the distances between celestial bodies are dynamically evolving, within the framework of action-angle variables, the Hamiltonian system can often be expressed as follows:
\begin{eqnarray}\label{EQ1}
H(I, \theta) = \varepsilon_0 H_0 (I) + \cdots + \varepsilon_\mathfrak{m} H_\mathfrak{m}(I) + \varepsilon P(I, \theta),
\end{eqnarray}
where $I\in D$, $D$ denotes a bounded closed region in $\mathds{R}^n$, $\theta\in \mathds{T}^n$, $0<\varepsilon\ll \min\limits_{1\leq i\leq \mathfrak{m}}\{\varepsilon_i\}\ll  1$, the parameters $\{\varepsilon_i\}$ need not follow a well-defined ordering. In such systems, a natural question arises: Does KAM stability persist under sufficiently small $\varepsilon$?  While we addressed this issue in \cite{Qian}, the non-degeneracy conditions presented there are in general and abstract form. This motivates the pursuit of a more flexible condition. In this paper, we will provide such a more convenient non-degenerate condition for the persistence of multi-scale invariant tori in Hamiltonian system (\ref{EQ1}).

Denote $\tilde{H}(I) = \varepsilon_0 H_0 (I) + \cdots + \varepsilon_\mathfrak{m} H_\mathfrak{m}(I)$ and $\omega = \partial_I \tilde{H} $. The equations of motion are given by
$\left\{
 \begin{array}{ll}
\dot{I}= \frac{\partial \tilde{H}}{\partial \theta} = 0, \\
\dot{\theta}= -\frac{\partial \tilde{H}}{\partial I} = \omega,
 \end{array}
 \right.$
which indicates that for a given initial condition $\big(I(0), \theta(0)\big) = (I_0, \theta_0)$ the solution of the equations of motion is
\begin{eqnarray*}
\left\{
                                                                      \begin{array}{ll}
                                                                        I(t)= I_0, \\
                                                                        \theta(t) = \theta_0 + \omega(I_0) t.
                                                                      \end{array}
                                                                    \right.
\end{eqnarray*}
If the frequency exhibits rational independence ($\langle k, \omega(I_0)\rangle \neq 0$ for any $k\in \mathds{Z}^n\setminus \{0\}$ ), the phase trajectories are everywhere dense on tori $ \mathds{T}^n\times \{I_0\}$, meaning the phase space is foliated by invariant torus $\mathds{T}^n \times \{I_0\}$. In contrast to the classical case, our study involves multi-scale frequencies, indicating that the flow speeds in different directions may operate on different scales.

The problem of the influence of small Hamiltonian perturbations on an integrable Hamiltonian system was called by Poincar\'{e} the fundamental problem of dynamics (\cite{Arnold1}). At present Poincar\'{e}'s `fundamental problem of the dynamics' continues to occupy one of the most important places in the theory of dynamical systems (\cite{Treschev}). Naturally, a key question arises: do the multi-scale invariant tori of the unperturbed system survive under small perturbations? In present paper, we will touch this problem.

The Hamiltonian system (\ref{EQ1}) frequently arises in various celestial mechanics problems, such as the spatial comet case, co-orbital motion in the three-body problem, and the spatial lunar problem \cite{Arnold4,Cors,Meyer3,Meyer4,Meyer5,Palacian}. The stability of these celestial mechanics problems is closely linked to the persistence of invariant tori in multi-scale Hamiltonian systems. This persistence originates from Arnold, who demonstrated the persistence of invariant tori for two-scale system $h_0 + \varepsilon h_1$ under the Kolmogorov and iso-energetic non-degenerate condition (\cite{Arnold4}). However, in many celestial mechanics scenarios, such as the spatial comet case and the spatial lunar problem, two scales are insufficient to eliminate degeneracy as pointed out in \cite{Cors,Meyer3,Meyer4,Meyer5,Palacian}.

Motivated by applications in a broader class of perturbed Kepler problems, \cite{han} proved the existence of full-dimensional invariant tori when there is an order relationship among the scales. However, there have been few studies on disordered multi-scale Hamiltonian systems. To address the dynamic stability of disordered multi-scale Hamiltonian systems, \cite{Qian} developed a multi-scale KAM iterative scheme and established a KAM theorem for such systems under general scenarios, thereby demonstrating `stability with respect to the measure of initial data'. Nevertheless, the hypothesis conditions in these results were formulated in general and abstract form. To facilitate the implementation of these studies in practical applications, it is necessary to propose more readily verifiable non-degeneracy conditions that better accommodate astronomical applications. For more information on the persistence of multi-scale invariant tori, refer to \cite{Fejoz,Xuemei,Qian1,Qian2,Qian3,Xu2,Xu3,Zhao}.

In this paper, we revisit the multi-scale Hamiltonian $(\ref{EQ1})$ and provide a more verifiable non-degenerate condition for the persistence of multi-scale invariant tori. To present our main results, we will first outline some assumptions.

\begin{itemize}
  \item [$\bf(R)$] There exists $N>0$ such that
  $$rank \{ \partial_I^\alpha \omega(I_0), 0\leq|\alpha| \leq n - 1\} = n, \forall I_0 \in D.$$
\end{itemize}
\begin{itemize}
\item[\bf{(K)}] Assume that there is an $n_1 \times n_1$ submatrix $\mathcal{N}$ of $(\partial_I^2 \tilde{H})^T$ such that
\begin{eqnarray*}
\mathcal{N}^T \mathcal{N} \geq c \tilde{\varepsilon}^2 I_{n_1\times n_1},
\end{eqnarray*}
where $c>0$ is a constant and $\tilde{\varepsilon} = \min\limits_{0\leq i\leq \mathfrak{m}} \{\varepsilon_i\}$.
\end{itemize}
\begin{itemize}
\item[\bf{(I)}] Assume that there is an $(n_1+1) \times (n_1+1)$ submatrix $\tilde{\mathcal{N}}$ of $\left(
  \begin{array}{cc}
    \partial_I^2 \tilde{H} & \omega^T \\
     \omega & 0 \\
  \end{array}
\right)$ such that
\begin{eqnarray*}
\tilde{\mathcal{N}}^T \tilde{\mathcal{N}} \geq c \tilde{\varepsilon}^2 I_{(n_1+1)\times (n_1+1)},
\end{eqnarray*}
where $c$ is a positive constant and $\tilde{\varepsilon} = \min\limits_{0\leq i\leq \mathfrak{m}} \{\varepsilon_i\}$.
\end{itemize}

We have the following theorem.

\begin{theorem}\label{MainTheorem}

Denote $\tilde{\varepsilon} = \min\limits_{0\leq i\leq \mathfrak{m}} \{\varepsilon_i\}$ and $\varepsilon = \epsilon\tilde{\varepsilon}$. Consider Hamiltonian $(\ref{EQ1})$ on $D\times \mathds{T}^n$.
\begin{enumerate}
  \item [\bf{1).}]
            Assume that condition $\bf{(R)}$ holds. Then, there exist $\varepsilon_0>0$ and a family of Cantor sets $D_\varepsilon \subset D$, $0< \varepsilon \leq \varepsilon_0$ , such that $|D\setminus D_\varepsilon| = O(\epsilon^{\frac{1}{N}})$. Each $I_0\in D_\varepsilon$ corresponds to a real analytic, invariant, quasi-periodic $n-$torus $T_{I_0}^\varepsilon=  \mathds{T}^n \times \{ I_0\}$ of Hamiltonian (\ref{EQ1}), which is a slight deformation of the unperturbed $n-$torus $T_{I_0}$, where the frequency is $\omega(I_0)$;
  \item [\bf{2).}] Assume that $\bf{(R)}$ and $\bf{(K)}$ hold. Then, there exist $\varepsilon_0>0$ and a family of Cantor sets  $D_\varepsilon \subset D$ for $0< \varepsilon \leq \varepsilon_0$, such that $|D\setminus D_\varepsilon| = O(\epsilon^{\frac{1}{N}})$. Each $I_0\in D_\varepsilon$ corresponds to a real analytic, invariant, quasi-periodic $n-$torus $T_{I_0}^\varepsilon$, which has $n_1$ components of frequency equal to $\omega(I_0)$;
  \item [\bf{3).}] Let $\Sigma = \{I: \tilde{H}(I) = c\}$ be a given energy surface. Assume that $\bf{(R)}$, $\bf{(K)}$ and $\bf{(I)}$ hold on $\Sigma$. Then, there exist $\varepsilon_0>0$ and a family of Cantor sets $\Sigma_\varepsilon \subset \Sigma$ for $0< \varepsilon \leq \varepsilon_0$, such that $|\Sigma\setminus \Sigma_\varepsilon| = O(\epsilon^{\frac{1}{N}})$. Each $I_0 \in \Sigma_\varepsilon$ corresponds to a real analytic, invariant, quasi-periodic $n-$torus, on which there are $n_1$ components of frequency $\omega_\varepsilon$ that satisfy $\omega_\varepsilon = t\omega(I_0)$, where $t \rightarrow1$ as $\varepsilon\rightarrow 0$.
\end{enumerate}

\end{theorem}

\begin{remark}
Condition $\textbf{(R)}$ ensures the existence of multi-scale invariant tori, and it is a more verifiable condition.
\end{remark}
\begin{remark}
Condition $\textbf{(K)}$ ensures the persistence of frequency components. When $n_1 = n$, condition $\textbf{(R)}$ is automatically satisfied. Specifically, when $n_1=n$, result \textbf{(2)} indicates that the multi-scale invariant tori with Diophantine frequency $\omega(I_0)$ for the unperturbed system persist under small perturbations.
\end{remark}

\begin{remark}
Conditions $\textbf{(K)}$ and $\textbf{(I)}$ together ensure the persistence of components of the frequency ratio. When $n_1 = n$, result \textbf{(3)} indicates that there exist perturbed tori for the perturbed system on a given energy surface. However, the frequency on the perturbed tori is $t \omega(I_0)$, not $\omega(I_0)$, where $t\rightarrow 0$ as $\varepsilon\rightarrow 0$.

\end{remark}

To prove Theorem \ref{MainTheorem}, we introduce the following parameterized multi-scale Hamiltonian system:
\begin{eqnarray}
\label{M2}\mathcal{H} (I, \theta, \xi) &=& \mathcal{N}(I, \xi, \varepsilon_i) + \mathcal{P}(I, \theta, \xi),\\
\nonumber \mathcal{N}(I, \xi, \varepsilon_i) &=& e(\xi,\varepsilon_i) + \langle \omega(\xi,\varepsilon_i), I\rangle + \frac{1}{2} \langle I, A(\xi,\varepsilon_i)I \rangle + \sum\limits_{3\leq |\jmath| \leq m-1} h_{\jmath} (\xi,\varepsilon_i) I^\jmath,\\
\nonumber e(\xi, \varepsilon_i) &=& \varepsilon_1 e^1(\xi) + \cdots + \varepsilon_\mathfrak{m} e^\mathfrak{m}(\xi), \\
\nonumber \omega(\xi, \varepsilon_i) &=& \varepsilon_1 \omega^1(\xi) + \cdots + \varepsilon_\mathfrak{m}\omega^\mathfrak{m}(\xi), \\
\nonumber A(\xi, \varepsilon_i) &=& \varepsilon_1 A^1(\xi) + \cdots + \varepsilon_\mathfrak{m} A^\mathfrak{m}(\xi),\\
\nonumber h_{\jmath,} (\xi, \varepsilon_i) &=& \varepsilon_1 h_{\jmath}^1(\xi) + \cdots + \varepsilon_\mathfrak{m} h_{\jmath}^\mathfrak{m}(\xi),\\
\nonumber |\mathcal{P}|_{r,s,h} &<& c \varepsilon,
\end{eqnarray}
where $(\theta, I) \in D_{r,s} = \{(I, \theta): |Im \theta| < s, |I|< r\}$, $\xi\in \Lambda \subseteq\mathds{R}^n$, $\Lambda$ represents a bounded closed region, $\varepsilon = \min\limits_{i}\{\varepsilon_i: 1\leq i \leq\mathfrak{m} \} \epsilon$, $0<\varepsilon, \epsilon, \varepsilon_i\ll1$, $|\cdot|_{r,s,h} = \sup\limits_{(I, \theta)\in D_{r,s}, \xi\in \Lambda_h} |\cdot|$, $\Lambda_h =\{\xi\in \mathds{C}^n: dist(\xi, \partial \Lambda)\leq h\}.$

To present the results for the parameterized Hamiltonian $(\ref{M2})$, we first outline some assumptions.

\begin{itemize}
  \item [$\bf(R')$] There exists $N>0$ such that
  $$rank \{ \partial_\xi^\alpha \omega, 0\leq|\alpha| \leq n - 1\} = n, \forall \xi \in \Lambda.$$
\end{itemize}
\begin{itemize}
\item[$\bf{(K')}$] Assume that there is an $n_1 \times n_1$ submatrix $\mathcal{A}$ of $A(\xi, \varepsilon_i)$ such that
\begin{eqnarray*}
\mathcal{A}^T \mathcal{A} \geq c \tilde{\varepsilon}^2 I_{n_1\times n_1},
\end{eqnarray*}
where $c>0$ is a constant and $\tilde{\varepsilon} = \min\limits_{0\leq i\leq \mathfrak{m}} \{\varepsilon_i\}$.
\end{itemize}
\begin{itemize}
\item[$\bf{(I')}$] Assume that there is an $(n_1+1) \times (n_1+1)$ submatrix $\tilde{\mathcal{A}}$ of $\left(
  \begin{array}{cc}
    A(\xi, \varepsilon_i) & \omega^T (\xi, \varepsilon_i)\\
     \omega(\xi, \varepsilon_i) & 0 \\
  \end{array}
\right)$ such that
\begin{eqnarray*}
\tilde{\mathcal{A}}^T \tilde{\mathcal{A}} \geq c \tilde{\varepsilon}^2 I_{(n_1+1)\times (n_1+1)},
\end{eqnarray*}
where $c$ is a positive constant and $\tilde{\varepsilon} = \min\limits_{0\leq i\leq \mathfrak{m}} \{\varepsilon_i\}$.
\end{itemize}

\begin{theorem}\label{Theorem1}
Denote $\tilde{\varepsilon} = \min\limits_{0\leq i\leq \mathfrak{m}} \{\varepsilon_i\}$ and $\varepsilon = \epsilon\tilde{\varepsilon}$.  Consider Hamiltonian $(\ref{M2})$ on $D_{r,s}\times \Lambda_h$.
\begin{enumerate}
  \item [\bf{1).}] Assuming condition $\bf{(R')}$ holds, there exist $\varepsilon_0>0$ and a family of Cantor sets $\Lambda_\varepsilon \subset \Lambda$ for $0< \varepsilon \leq \varepsilon_0$, such that $|\Lambda\setminus \Lambda_\varepsilon| = O(\epsilon^{\frac{1}{N}})$. Additionally, there exists a symplectic transformation $\Phi:$ $D(\frac{r}{2}, \frac{s}{2}) \times \Lambda_\varepsilon \rightarrow D(r,s) \times \Lambda$ that is close to the identity, satisfying
\begin{eqnarray*}
H\circ \Phi = e_*+ \langle \omega_*, y\rangle + \langle y, A_*y\rangle + \mathcal{P}_*(x,y),
\end{eqnarray*}
where $\mathcal{P}_* = \mathcal{P}_*(x,y)= O(|y|^3)$. Consequently, for $\xi \in \Lambda_\varepsilon$, the unperturbed tori $T_{\xi}=\mathds{T}^n \times \{\xi\}$ with toral frequency $\omega$ persist, leading to an invariant torus of the perturbed system with the toral frequency $\omega_*$;
  \item [\bf{2).}]Assuming $(R')$ and $(K')$ hold on $\Lambda$, there exist $\varepsilon_0>0$ and a family of Cantor sets $\Lambda_\varepsilon \subset \Lambda$ for $0< \varepsilon \leq \varepsilon_0$, such that $|\Lambda\setminus \Lambda_\varepsilon| = O(\varepsilon^{\frac{1}{N}})$. Additionally, there is a symplectic transformation $\Phi:$ $D(\frac{r}{2}, \frac{s}{2}) \times \Lambda_\varepsilon \rightarrow D(r,s) \times \Lambda$ that is close to the identity, satisfying
\begin{eqnarray*}
H\circ \Phi = e_*+ \langle \omega_*, y\rangle + \langle y, A_*y\rangle + \mathcal{P}_*(x,y),
\end{eqnarray*}
where $\mathcal{P}_* = \mathcal{P}_*(x,y)= O(|y|^3)$. Moreover, we have
\begin{eqnarray*}
\big(\omega_*\big) _{i_q} = \big(\omega \big) _{i_q}, 1\leq q \leq n_1.
\end{eqnarray*}
Therefore, for $\xi \in \Lambda_\varepsilon$, the unperturbed tori $T_{\xi}= \mathds{T}^n \times \{\xi\}$ with toral frequency $\omega_*$ persist and give rise to an invariant torus of the perturbed system with the toral frequency $\omega_*$, which preserves the frequency components  $\omega_{i_1}$ $\cdots$, $\omega_{i_{n_1}}$ of the unperturbed toral frequency $\omega$;

  \item [\bf{3).}]
Let $\Sigma = \{\xi: \tilde{N}(\xi) = c\}$ be a given energy surface. Assume that $(R')$, $(K')$ and $(I')$ hold on $\Sigma$. Then, there exist $\varepsilon_0>0$, a family of Cantor sets $\Sigma_\varepsilon \subset \Sigma$ for $0< \varepsilon \leq \varepsilon_0$, such that $|\Sigma\setminus \Sigma_\varepsilon| = O(\varepsilon^{\frac{1}{N}})$, and a symplectic transformation $\Phi:$ $D(\frac{r}{2}, \frac{s}{2}) \times \Sigma_\varepsilon \rightarrow D(r,s)  \times \Sigma$, which is close to the identity such that
\begin{eqnarray*}
H\circ \Phi = e_*+ \langle \omega_*, y\rangle + \langle y, A_*y\rangle + \mathcal{P}_*(x,y),
\end{eqnarray*}
where $\mathcal{P}_* = \mathcal{P}_*(x,y)= O(|y|^3)$. Moreover,
\begin{eqnarray*}
\big[\omega_{*, i_1}: \cdots: \omega_{*, i_{n_1}}\big] = \big[\omega_{i_1}: \cdots: \omega_{i_{n_1}}\big].
\end{eqnarray*}
Therefore, for $\xi \in \Sigma_\varepsilon$, the unperturbed tori $T_{\xi}=\mathds{T}^n \times \{\xi\}$ with toral frequency $\omega(\xi)$ persist and give rise to an invariant torus of the perturbed system with toral frequency $\omega_*$, which preserves the frequency ratio $[\omega_{i_1} : \cdots : \omega_{i_{n_1}} ]$ of the unperturbed toral frequency $\omega(\xi)$.
\end{enumerate}

\end{theorem}

This paper is structured as follows. We prove Theorem \ref{Theorem1} in Section \ref{ProofofTH2}. Specifically, the proof is divided into four parts:
\begin{itemize}
  \item \textbf{KAM Step}: For the parameterized multi-scale Hamiltonian system, we introduce a KAM step that involves truncating the perturbation, extending the small divisor estimate to a larger region, solving the homological equation, and performing frequency and coordinate transformations, while estimating the new error term.
  \item \textbf{Iteration Lemma}: We construct a family of iterative sequences that ensure the KAM step can be applied infinitely.
  \item \textbf{Convergence}: We demonstrate that the composition of an infinite series of nearly identity symplectic transformations converges within the iterative sequence.
  \item \textbf{Measure Estimate}: In the KAM step, we resolve the homological equation on a set of nearly full measure and show that after infinitely many iterations, the remaining set also has nearly full measure.
\end{itemize}
In Section \ref{SEC3}, we prove Theorem \ref{MainTheorem} using Theorem \ref{Theorem1}. Finally, in Section \ref{Example}, we provide an example inspired by \cite{Cors1}.

\section{Proof of Theorem \ref{Theorem1}}\label{ProofofTH2}

Let $\Lambda$ represent a bounded closed region in $\mathds{R}^n$ and $\Lambda_h =\{\xi\in \mathds{C}^n: dist(\xi, \partial \Lambda)\leq h\}.$ Recall $\tilde{\varepsilon} = \min\limits_{0\leq i\leq \mathfrak{m}} \{\varepsilon_i\}$. For fixed $\gamma>0$ and $\tau\geq n-1$, let
\begin{eqnarray*}
\Omega_{\gamma, \tau}= \{\xi\in \Lambda: |\langle k, \omega(\xi,\varepsilon_i)\rangle| \geq \frac{\tilde{\varepsilon}\gamma}{|k|^\tau}, 0\neq k \in \mathds{Z}^n\},
\end{eqnarray*}
where $|k|= |k_1|+ \cdots + |k_n|.$ Denote $\Omega_{\gamma, \tau}^\gamma = \{\xi\in \Omega_{\gamma, \tau}: dist (\xi, \partial \Lambda) < \gamma\}$, $|\cdot|_{r,s} = \sup\limits_{(I, \theta)\in D_{r,s}} |\cdot|$ and $|\cdot|_{r,s,h} = \sup\limits_{(I, \theta)\in D_{r,s}, \xi\in \Lambda_h} |\cdot|$.

Consider the following real analytic parameterized multi-scale Hamiltonian system:
\begin{eqnarray}
\label{EQ2}\mathcal{H} (I, \theta, \xi) &=& \mathcal{N}(I, \xi, \varepsilon_i) + \mathcal{P}(I, \theta, \xi),\\
\nonumber \mathcal{N}(I, \xi, \varepsilon_i) &=& e(\xi,\varepsilon_i) + \langle \omega(\xi,\varepsilon_i), I\rangle + \frac{1}{2} \langle I, A(\xi,\varepsilon_i)I \rangle + \sum\limits_{3\leq |\jmath| \leq m-1} h_{\jmath} (\xi,\varepsilon_i) I^\jmath,\\
\nonumber e(\xi, \varepsilon_i) &=& \varepsilon_1 e^1(\xi) + \cdots + \varepsilon_\mathfrak{m} e^\mathfrak{m}(\xi), \\
\nonumber \omega(\xi, \varepsilon_i) &=& \varepsilon_1 \omega^1(\xi) + \cdots + \varepsilon_\mathfrak{m}\omega^\mathfrak{m}(\xi), \\
\nonumber A(\xi, \varepsilon_i) &=& \varepsilon_1 A^1(\xi) + \cdots + \varepsilon_\mathfrak{m} A^\mathfrak{m}(\xi),\\
\nonumber h_{\jmath,} (\xi, \varepsilon_i) &=& \varepsilon_1 h_{\jmath}^1(\xi) + \cdots + \varepsilon_\mathfrak{m} h_{\jmath}^\mathfrak{m}(\xi),\\
\nonumber |\mathcal{P}|_{r,s,h} &<& c \varepsilon,
\end{eqnarray}
where $(\theta, I) \in D_{r,s} = \{(I, \theta): |Im \theta| < s, |I|< r\}$, $\xi\in \Lambda \subseteq\mathds{R}^n$, $\Lambda_h =\{\xi\in \mathds{C}^n: dist(\xi, \partial \Lambda)\leq h\}$, $\varepsilon = \tilde{\varepsilon} \epsilon$, $0<\varepsilon, \epsilon, \varepsilon_i\ll1$, $|\cdot|_{r,s,h} = \sup\limits_{(I, \theta)\in D_{r,s}, \xi\in \Lambda_h} |\cdot|$.

\subsection{A KAM Step}\label{AKAMstep}

 We summarize a KAM step into the following proposition.
\begin{pro}
Consider Hamiltonian (\ref{EQ2}). Assume
\begin{itemize}
  \item [(a)] $K > \frac{n-1}{\sigma}$,
  \item [(b)] $\int_{K}^{\infty} x ^{n-1} e^{-x \sigma} dx < \eta^m$,
  \item [(c)] $\hat{\varepsilon} h  K^{\tau+1} = o(\gamma \tilde{\varepsilon}),$
  \item [(d)] $r K^{\tau+1} = o(\gamma \tilde{\varepsilon})$,
  \item [(e)]$\epsilon < \gamma r^2 \eta^m \sigma^{\tau + 1}$.
\end{itemize}
 Then there exists a transformation:
\begin{eqnarray*}
\mathcal{F} = \Phi\circ \varphi = (U, V, \varphi):  D_{s-5\sigma, \eta r }\times\Lambda_{\eta h}  \rightarrow D_{s, r}\times \Lambda_{h},
\end{eqnarray*}
such that
\begin{eqnarray*}
\mathcal{H}\circ \mathcal{F} = \mathcal{N}_+ + \mathcal{P}_+,
\end{eqnarray*}
where
\begin{eqnarray*}
\mathcal{N}_+ (I, \xi,\varepsilon_i) &=&e_+(,\varepsilon_i) + \langle \omega(\xi,\varepsilon_i), I\rangle + \frac{1}{2} \langle I, A_+(\xi,\varepsilon_i)I \rangle + \sum\limits_{3\leq |\jmath| \leq m-1} h_{\jmath} (\xi,\varepsilon_i) I^\jmath,\\
|\mathcal{P}_+|_{\eta r, s- 5\sigma, \eta h} &<& \varepsilon_+< \gamma r_+^2 \eta_+^m \sigma_+^{\tau + 1}.
\end{eqnarray*}
Moreover,
\begin{eqnarray*}
|U - id |_{\frac{r}{2}, s- 3\sigma} &<& \frac{\varepsilon}{\tilde{\varepsilon} \gamma \sigma^{n+1}},  |V- id|_{\frac{r}{2}, s- 3\sigma} < \frac{\varepsilon}{\tilde{\varepsilon} \gamma r \sigma^\tau},\\
|U_I - I|_{\eta r, s- 3\sigma} &<& \frac{\varepsilon}{\tilde{\varepsilon} \gamma r \sigma^{\tau+1}}, |V_\theta - I|_{\frac{r}{2}, s- 4\sigma} < \frac{\varepsilon}{ \tilde{\varepsilon} \gamma r \sigma^{\tau+1}}, \\
|U_\theta|_{\frac{r}{2}, s- 4\sigma} &<& \frac{\varepsilon}{ \tilde{\varepsilon} \gamma \sigma^{\nu+1}}, |V_I|_{\eta r, s- 3\sigma}< \frac{\varepsilon}{\tilde{\varepsilon} \gamma r^2 \sigma^{\tau+1}},\\
 |\varphi - id|_{\frac{h}{2}} &<& \frac{\varepsilon}{r}, |D\varphi - I|_{\frac{h}{4}}< \frac{\varepsilon}{r},\\
 |\partial_\xi (U - id )|_{\frac{r}{2}, s- 3\sigma, \frac{h}{2}}&<& \frac{\varepsilon}{ \tilde{\varepsilon} \gamma h \sigma^{n+1}},
 |\partial_\xi (V- id)|_{\frac{r}{2}, s- 3\sigma, \frac{h}{2}} < \frac{\varepsilon}{\tilde{\varepsilon} \gamma r h \sigma^\tau}.
\end{eqnarray*}
\end{pro}

We divide the proof of this Proposition into six parts: truncation, extending the small divisor estimate, the homological equation, frequency transformation, coordinate transformation, and the new error term.

\subsubsection{Truncation}
Consider the Taylor-Fourier series of $\mathcal{P}(I, \theta, \xi):$ $\mathcal{P}(I, \theta, \xi) = \sum\limits_{\jmath\in \mathds{Z}_+^n, k\in \mathds{Z}^n} P_{k\jmath} I^\jmath e^{i \langle k,  \theta\rangle}.$
Denote $K$ a sufficient large number. Let $Q(I, \theta, \xi) = \sum\limits_{|\jmath| \leq m-1, k\in \mathds{Z}^n} P_{k\jmath} I^\jmath e^{i \langle k,  \theta\rangle},$ $R(I, \theta, \xi) = \sum\limits_{|\jmath| \leq \jmath, |k|\leq K} P_{k\jmath} I^\jmath e^{i \langle k,  \theta\rangle}.$

\begin{lemma}
Assume
\begin{enumerate}
  \item [$(a)$] $K > \frac{n-1}{\sigma}$,
  \item [$(b)$] $\int_{K}^{\infty} x ^{n-1} e^{-x s} dx < \eta^m$.
\end{enumerate}
Then there is a constant $c= c(n)>0$ such that
\begin{eqnarray}
|\mathcal{P}-R|_{2\eta r, s-\sigma} &<& c(n)\eta^m \varepsilon,\\
|R|_{r, s-\sigma}&<& c(n) \varepsilon.
\end{eqnarray}
\end{lemma}
\begin{proof}
Obviously, $|Q|_{r,s} < \varepsilon$.  According to Cauchy's estimate, we have
\begin{eqnarray*}
|\mathcal{P}-Q|_{2\eta r ,s} &=& | \sum\limits_{|\jmath| \geq m, k\in \mathds{Z}^n} P_{k\jmath} I^\jmath e^{i \langle k,  \theta\rangle}|_{2\eta r ,s}\\
&\leq& \sum\limits_{ |\jmath| \geq m} \frac{(2 \eta r)^\jmath}{(r - 2 \eta r)^\jmath} | \mathcal{P} |_{r,s} \leq \eta^m \varepsilon.
\end{eqnarray*}
Then, combining assumptions $(a)$ and $(b)$, we obtain
\begin{eqnarray*}
|Q-R|_{r, s-\sigma}&=& |\sum\limits_{|\jmath| \leq m-1, |k| > K} P_{k\jmath} I^\jmath e^{i \langle k,  \theta\rangle}|_{r,s-\sigma}\\
&\leq&  \sum\limits_{|k| > K} |\mathcal{P}|_{r,s} e^{-|k|\sigma}\\
&\leq& \int_{K}^{\infty} 4^n x ^{n-1} e^{-x \sigma} dx |\mathcal{P}|_{r,s} \leq |\mathcal{P}|_{r,s} \eta^m.
\end{eqnarray*}
Therefore,
\begin{eqnarray*}
|\mathcal{P}-R|_{2\eta r, s-\sigma} &<& |\mathcal{P}-Q|_{2\eta r ,s}+ |Q-R|_{r, s-\sigma}< c(n)\eta^m \varepsilon.
\end{eqnarray*}
Using assumption (b), we have
\begin{eqnarray*}
|R|_{r, s-\sigma} &<& c(n) \varepsilon.
\end{eqnarray*}

\end{proof}

\subsubsection{Extending the Small Divisor Estimate.}
The nonresonant conditions are assumed to hold only on the real set $\Omega_{\gamma,\tau}$. The next lemma demonstrates that these nonresonant conditions also hold on $\Lambda_h$.

\begin{lemma}
Assume
\begin{itemize}
  \item [$(c)$] $\hat{\varepsilon} h  K^{\tau+1} = o(\gamma \tilde{\varepsilon}).$
\end{itemize}
Then, on $O_h$, $|\langle k, \omega(\xi, \varepsilon_i)\rangle| > \frac{\tilde{\varepsilon} \gamma}{2|k|^\tau}.$
\end{lemma}

\begin{proof}
For every $\xi\in \Lambda_h$, there is a $\xi_* \in \Omega_{\gamma,\tau}$, such that
\begin{eqnarray*}
|\langle k, \omega(\xi, \varepsilon_i)\rangle| &\geq& \big||\langle k, \omega (\xi_*, \varepsilon_i)\rangle| - |\langle k, \omega(\xi, \varepsilon_i)- \omega (\xi_*)\rangle|\big| \\
&>& \frac{\tilde{\varepsilon}\gamma}{|k|^\tau} - K h \hat{\varepsilon} >\frac{\tilde{\varepsilon} \gamma}{2|k|^\tau},
\end{eqnarray*}
provided $\hat{\varepsilon} h  K^{\tau+1} = o(\gamma \tilde{\varepsilon})$, where $\hat{\varepsilon} = \max\limits_{1\leq i\leq m} \{ \varepsilon_i\}$.
\end{proof}

\subsubsection{Homological Equation}
Let $F= \sum\limits_{0<|k|\leq K, |\jmath| \leq m-1} F_{k\jmath} I^\jmath e^{i \langle k, \theta\rangle}.$
\begin{lemma}
Denote $[R] = \int_{\mathds{T}^n} R dx.$ Assume
\begin{enumerate}
  \item [$(d)$] $r K^{\tau+1} = o(\gamma \tilde{\varepsilon})$.
\end{enumerate}
There is a symplectic transformation $\phi_F^1$ generated by $F$ such that $(\ref{EQ2})$ is transformed to
\begin{eqnarray*}
\mathcal{H}\circ \phi_F^1 &=& \tilde{e}_+(\xi, \varepsilon_i)+ \langle \tilde{\omega}_+(\xi, \varepsilon_i), I\rangle+ \frac{1}{2} \langle I, \tilde{A}_+ (\xi, \varepsilon_i)I\rangle+ \tilde{\hat{h}}_+(\xi, \varepsilon_i)+ \tilde{\mathcal{P}}_+,
\end{eqnarray*}
where
\begin{eqnarray*}
\tilde{e}_+ &=& e + P_{00},  \tilde{\omega}_+ =  \omega + P_{01}, \tilde{A}_+ = A + P_{02},\\
\tilde{\hat{h}}_+ &=& \sum\limits_{3\leq |\jmath| \leq m-1} h_{\jmath} (\xi,\varepsilon_i) I^\jmath  + \sum\limits_{3\leq |\jmath| \leq \jmath} P_{0\jmath} I^\jmath,\\
\tilde{\mathcal{P}}_+ &=& \int_0^1 \{(1-t)[R]+ t R, F\}\circ \phi_F^t dt+ (\mathcal{P} - R)\circ \phi_F^1.
\end{eqnarray*}
Moreover,
\begin{eqnarray}\label{EQ3}
|F|_{r, s-2\sigma}&\leq& \frac{\epsilon}{ \gamma \sigma^\tau}.
\end{eqnarray}
\end{lemma}

\begin{proof}
Let $\phi_F^t$ denote the flow generated by $F$.  Suppose
\begin{eqnarray}\label{HomoEqu}
\{\mathcal{N}, F\} + R= [R].
\end{eqnarray}
Then
\begin{eqnarray}
\nonumber&~&\mathcal{H}\circ \phi_F^1 \\
\nonumber &=& \mathcal{N}\circ \phi_F^1 + R\circ \phi_F^1 + (\mathcal{P}-R)\circ \phi_F^1 \\
\label{TransSys}&=& \mathcal{N} + [R]+  \int_0^1 \{(1-t)\{\mathcal{N}, F\}+ R, F\}\circ \phi_F^t dt +(\mathcal{P}-R)\circ \phi_F^1.
\end{eqnarray}

Let
\begin{eqnarray*}
\tilde{\mathcal{P}}_+ &=& \int_0^1 \{(1-t)[R]+ t R, F\}\circ \phi_F^t dt+ (\mathcal{P} - R)\circ \phi_F^1.
\end{eqnarray*}
Then $(\ref{TransSys})$ is changed to
\begin{eqnarray}\label{EQ3}
\nonumber \mathcal{H}\circ \phi_F^1 &=& \mathcal{N} + [R] + \tilde{\mathcal{P}}_+\\
&=& \tilde{e}_+(\xi, \varepsilon_i)+ \langle \tilde{\omega}_+(\xi, \varepsilon_i), I\rangle+ \frac{1}{2} \langle I, \tilde{A}_+ (\xi, \varepsilon_i)I\rangle+ \tilde{\hat{h}}_+(\xi, \varepsilon_i),
\end{eqnarray}
where
\begin{eqnarray*}
\tilde{e}_+(\xi, \varepsilon_i) &=& e (\xi, \varepsilon_i)+ P_{00}(\xi),  \tilde{\omega}_+(\xi, \varepsilon_i) =  \omega(\xi, \varepsilon_i) + P_{01}(\xi),\\
\tilde{A}_+(\xi, \varepsilon_i) &=& A (\xi, \varepsilon_i)+ P_{02}(\xi), \\
\tilde{\hat{h}}_+(\xi, \varepsilon_i) &=& \sum\limits_{3\leq |\jmath| \leq m-1} h_{\jmath} (\xi, \varepsilon_i) I^\jmath  + \sum\limits_{3\leq |\jmath| \leq \jmath} P_{0\jmath} (\xi) I^\jmath.
\end{eqnarray*}

Directly,
\begin{eqnarray*}
\{\mathcal{N}, F\}&=& \frac{\partial \mathcal{N}}{\partial \theta} \frac{\partial F}{\partial I} - \frac{\partial \mathcal{N}}{\partial I} \frac{\partial F}{\partial \theta}=- \frac{\partial \mathcal{N}}{\partial I} \frac{\partial F}{\partial \theta}\\
&=&- \sum\limits_{0<|k|\leq K} \langle k, \omega + A(\xi,\varepsilon_i)I + \partial_I(\sum\limits_{|\jmath|=3}^{m-1} h_\jmath(\xi,\varepsilon_i) I^\jmath)\rangle F_{k\jmath} e^{i \langle k, \theta\rangle}.
\end{eqnarray*}
Using $(\ref{HomoEqu})$, we have
\begin{eqnarray}\label{EQ5}
- \sum\limits_{0<|k|\leq K} \langle k, \omega + A I + \partial_I(\sum\limits_{|\jmath|=3}^{m-1} h_\jmath I^\jmath)\rangle F_{k\jmath} e^{i \langle k, \theta\rangle}= \sum\limits_{0<|k|\leq K} P_{k\jmath} e^{i \langle k, \theta\rangle}.
\end{eqnarray}
By comparing the coefficients on both sides of $(\ref{EQ5})$, we obtain:
 \begin{eqnarray*}
 F_{k\jmath} = \frac{P_{k\jmath}}{\langle k, \omega + A(\xi, \varepsilon_i)I + \partial_I(\sum\limits_{|\jmath|=3}^{m-1} h_\jmath(\xi, \varepsilon_i) I^\jmath)\rangle}.
 \end{eqnarray*}

Let $\Delta_{\gamma, \tau} = \{\xi: |\langle k, \omega(\xi, \varepsilon_i)\rangle| \geq \frac{\tilde{\varepsilon}\gamma}{|k|^\tau}, 0<|k|\leq K\}.$ Then, using assumption $(d)$, for  $\xi\in \Delta_{\gamma, \tau}$,
\begin{eqnarray*}
&~&|\langle k, \omega(\xi, \varepsilon_i) + A(\xi, \varepsilon_i)I + \partial_I(\sum\limits_{|\jmath|=3}^{m-1} h_\jmath(\xi, \varepsilon_i) I^\jmath)\rangle| \\
&\geq& \big| |\langle k, \omega(\xi, \varepsilon_i)\rangle| - |\langle k, A(\xi, \varepsilon_i)I + \partial_I(\sum\limits_{|\jmath|=3}^{m-1} h_\jmath(\xi, \varepsilon_i) I^\jmath)\rangle| \big|\\
&\geq& \frac{\tilde{\varepsilon}\gamma}{ |k|^\tau} - |k|r\hat{\varepsilon} \geq \frac{\hat{\varepsilon}\gamma}{ 2|k|^\tau}.
\end{eqnarray*}
Hence
\begin{eqnarray*}
&~&|F|_{r, s-2\sigma} \\
&=& |\sum\limits_{0<k \leq K, 0<|\jmath|\leq m-1} F_{k\jmath} I^\jmath e^{i \langle k, \theta\rangle}|_{r, s-2\sigma}\\
&=&|\sum\limits_{0<k \leq K, 0<|\jmath|\leq m-1} \frac{P_{k\jmath}}{\langle k, \omega + A(\xi, \varepsilon_i)I + \partial_I(\sum\limits_{|\jmath|=3}^{m-1} h_\jmath(\xi, \varepsilon_i) I^\jmath)\rangle} I^\jmath e^{i \langle k, \theta\rangle}|_{r, s-2\sigma}\\
&\leq& |\sum\limits_{0<k \leq K, 0<|\jmath|\leq m-1} \frac{2 |k|^\tau P_{k\jmath}}{\hat{\varepsilon} \gamma} I^\jmath e^{i \langle k, \theta\rangle}|_{r, s-2\sigma}\\
&\leq& \frac{|R|_{r, s-\sigma}}{\hat{\varepsilon} \gamma \sigma^\tau} \leq \frac{\varepsilon}{\hat{\varepsilon} \gamma \sigma^\tau} \leq \frac{\epsilon}{ \gamma \sigma^\tau}.
\end{eqnarray*}
\end{proof}

\subsubsection{Frequency Transformation}
The following lemma demonstrates  the partial preservation of frequency.
\begin{lemma} \label{LM2}
 Assume $\bf{(K')}$.
Then there is a symplectic transformation:
\begin{eqnarray*}
\phi: I \rightarrow I + I^*, ~~ \theta \rightarrow \theta,
\end{eqnarray*}
such that
\begin{eqnarray*}
\mathcal{H}_+ (I,\theta,\xi, \varepsilon_i)&=& \mathcal{H}\circ \phi_F^1 \circ \phi  \\
&=& e_+(\xi, \varepsilon_i) + \langle\omega_+(\xi, \varepsilon_i), I\rangle+ {h}_+(\xi, I, \varepsilon_i) + \mathcal{P}_+(\xi),
\end{eqnarray*}
where
\begin{eqnarray*}
e_+(\xi, \varepsilon_i)&=& \tilde{e}(\xi, \varepsilon_i) + \langle \tilde{\omega}(\xi, \varepsilon_i), I^*\rangle + \frac{1}{2} \langle I^* , \tilde{A}(\xi, \varepsilon_i) I^*\rangle + [\mathcal{R}](I^*, \xi, \varepsilon_i),\\
\omega_+(\xi, \varepsilon_i) &=& {\omega}(\xi, \varepsilon_i) + \left(
                        \begin{array}{c}
                          0 \\
                         D_1 P_{01}+ \bar{P}_{01} \\
                        \end{array}
                      \right)
,\\
A_+(\xi, \varepsilon_i) &=& \tilde{A} (\xi, \varepsilon_i) + \varepsilon \partial_I^2 [\mathcal{R}](I^*, \xi, \varepsilon_i),\\
h_+(\xi, \varepsilon_i) &=& \langle I, \tilde{A}_+(\xi, \varepsilon_i) I\rangle,\\
\mathcal{P}_+ (\xi)&=& \tilde{\mathcal{P}}_+\circ \phi.
\end{eqnarray*}
(The definitions of $D_1$, $P_{01}$ and $\bar{P}_{01}$ are shown in the proof.) Moreover, $$| I_*| \leq \frac{\epsilon}{r},| \partial_\xi I_*| \leq \frac{\epsilon}{r h}.$$

\end{lemma}
\begin{proof}
According to condition $\bf{(K')}$, there exists an $n_1 \times n_1$ nonsingular submatrix for $A$. Thus, there is an orthogonal matrix $O$ such that $$O^T A O = \left(
                                                                                     \begin{array}{cc}
                                                                                       \tilde{A} & \tilde{B} \\
                                                                                       \tilde{C} & \tilde{D} \\
                                                                                     \end{array}
                                                                                   \right),
$$
where $rank (\tilde{A}, \tilde{B}) = n_1$. Furthermore, $\tilde{A}$ is non-singular, since $A$ is symmetrical.

Let $$ A I  =-\varepsilon P_{01}.$$ Then $$O^TAOO^T I  =- \varepsilon O^T P_{01}.$$ Next, denote $\tilde{I}$, $\bar{I}$, $P_{01}$ and $\bar{P}_{01}$ as the first $n_1$ and the last $l-n_1$ components of $O^TI$ and $O^T P_{01}$, respectively. Therefore, we have
\begin{eqnarray}\label{604}
\left(
                                                                                     \begin{array}{cc}
                                                                                       \tilde{A} & \tilde{B} \\
                                                                                       \tilde{C} & \tilde{D} \\
                                                                                     \end{array}
                                                                                   \right) \left(
                                                                                             \begin{array}{c}
                                                                                               \tilde{I} \\
                                                                                               \bar{I} \\
                                                                                             \end{array}
                                                                                           \right) = - \varepsilon\left(
                                                                                                             \begin{array}{c}
                                                                                                                P_{01} \\
                                                                                                                \bar{P}_{01} \\
                                                                                                             \end{array}
                                                                                                           \right)
                                                                                     .
\end{eqnarray}

Let $$T_1  = \left(
                        \begin{array}{cc}
                          I & 0 \\
                          D_1 & I \\
                        \end{array}
                      \right),
$$
where $D_1$ is determined by the linear relationship among the rows of $(\tilde{C}, \tilde{D})$ and $(\tilde{A}, \tilde{B})$, $I$ is identity matrix. Thus, equation $(\ref{604})$ is equivalent to
\begin{eqnarray}\label{Eq11}
\left(
                                                                                     \begin{array}{cc}
                                                                                       \tilde{A} & \tilde{B }\\
                                                                                       0   & 0 \\
                                                                                     \end{array}
                                                                                   \right) \left(
                                                                                             \begin{array}{c}
                                                                                               \tilde{I} \\
                                                                                               \bar{I} \\
                                                                                             \end{array}
                                                                                           \right) = - \varepsilon\left(
                                                                                                             \begin{array}{c}
                                                                                                                P_{01} \\
                                                                                                                D_1 P_{01} + \bar{P}_{01} \\
                                                                                                             \end{array}
                                                                                                           \right)
                                                                                     .
\end{eqnarray}

Redefine $I^*  =  \left(
                                                                             \begin{array}{c}
                                                                               \tilde{I} \\
                                                                               0 \\
                                                                             \end{array}
                                                                           \right)
$. Then $I^*$ is a special solution of the following equation:
\begin{eqnarray}\label{Eq10}
\left(
                                                                                     \begin{array}{cc}
                                                                                       \tilde{A} & \tilde{B} \\
                                                                                       0   & 0 \\
                                                                                     \end{array}
                                                                                   \right) \left(
                                                                                             \begin{array}{c}
                                                                                               \tilde{I} \\
                                                                                               \bar{I} \\
                                                                                             \end{array}
                                                                                           \right) = - \varepsilon\left(
                                                                                                             \begin{array}{c}
                                                                                                                P_{01} \\
                                                                                                               0    \\
                                                                                                             \end{array}
                                                                                                           \right)
                                                                                     .
\end{eqnarray}
Therefore, under the symplectic transformation $\phi: \theta\rightarrow \theta,~I\rightarrow I+ I^*,$ we obtain
\begin{eqnarray*}
\mathcal{H}_+&=& \mathcal{H}\circ \phi_F^1 \circ \phi = e_+ + \langle\omega_+, y\rangle+ {h}_+(y) + \mathcal{P}_+,
\end{eqnarray*}
where
\begin{eqnarray*}
e_+&=& \tilde{e}_+ + \langle \tilde{\omega}_+(\xi), I_*\rangle + \frac{1}{2} \langle I_* , \tilde{A}_+ I_*\rangle + [\mathcal{R}](I_*),\\
\omega_+ &=& \omega + \left(
                        \begin{array}{c}
                          0 \\
                         D_1 P_{01}+ \bar{P}_{01} \\
                        \end{array}
                      \right)
,\\
A_+ &=& A + \varepsilon \partial_y^2 [\mathcal{R}](y^*),\\
h_+ &=& \langle y, A_+ y\rangle,\\
\mathcal{P}_+&=& \tilde{\mathcal{P}}_+\circ \phi.
\end{eqnarray*}
Substituting $I^*$ into $(\ref{Eq11})$, we find $$\tilde{A} \tilde{I} = -\varepsilon P_{01}.$$ Thus
 $$\tilde{I}^T \tilde{A}^T \tilde{A} \tilde{I} = -\varepsilon^2 P_{01}^T P_{01}.$$
 Combining with condition $\bf(K')$, we have $$\frac{\tilde{I}^T}{|\tilde{I}|} (\tilde{A}^T \tilde{A} - c\min\limits_{0\leq i\leq m}\{ \varepsilon_i^2 \} I_{n_1 \times n_1}) \frac{\tilde{I}}{|\tilde{I}|} |\tilde{I}|^2 \geq 0.$$ Therefore,
 \begin{eqnarray*}
 \min\limits_{0\leq i\leq m}\{ \varepsilon_i^2\}  |\tilde{I}|^2  &=& \frac{\tilde{I}^T}{|\tilde{I}|} \min\limits_{0\leq i\leq m}\{ \varepsilon_i^2 \} I_{n_1 \times n_1} \frac{\tilde{I}}{|\tilde{I}|}|\tilde{I}|^2\\
 &\leq& \frac{\tilde{I}^T}{|\tilde{I}|} {\tilde{A}}^T \tilde{A} \frac{\tilde{I}}{|\tilde{I}|}|\tilde{I}|^2 = \varepsilon^2 P_{01}^T P_{01},
 \end{eqnarray*}
 i.e.$$|\tilde{I}|_{\max} \leq \sqrt{|\tilde{I}|^2} \leq \sqrt{\frac{\varepsilon^2}{\min\limits_{0\leq i\leq m}\{ \varepsilon_i^2\}} P_{01}^T P_{01}} \leq \sqrt{P_{01}^T P_{01}}. $$
Hence, $|\tilde{I}| \leq c \frac{\epsilon}{r}.$ By Cauchy's estimate, we obtain $|\partial_\xi I^*|\leq c \frac{\epsilon}{r h}$.
\end{proof}

The next lemma addresses the partial preservation of the frequency ratio on a given energy surface.
\begin{lemma}\label{LM3}
Assume $\bf{(K')}$ and $\bf{(I')}$.
Then there is a symplectic transformation:
\begin{eqnarray*}
\phi: I \rightarrow I + I_1^*, ~~ \theta \rightarrow \theta,
\end{eqnarray*}
such that
\begin{eqnarray*}
\mathcal{H}_+&=& \tilde{\mathcal{H}}_+ \circ \phi = e_+ + \langle\omega_+, I\rangle+ {h}_+(I) + \mathcal{P}_+,
\end{eqnarray*}
where
\begin{eqnarray*}
e_+&=& e, ~\omega_+ = (1-t)\omega + \left(
                        \begin{array}{c}
                          0 \\
                         D_2 \breve{P}+ \bar{P}_{01} \\
                        \end{array}
                      \right)
,\\
A_+ &=& A + \varepsilon \partial_I^2 [\mathcal{R}](I_1^*),h^+ = \langle I, A_+ I\rangle,\mathcal{P}_+= \tilde{\mathcal{P}}_+\circ \phi,
\end{eqnarray*}
(The definitions of $D_2$, $\breve{P}$ and $\bar{P}_{01}$ are shown in the proof.) Moreover, $$|I_*| \leq \frac{\epsilon}{r},| \partial_\xi I_*| \leq \frac{\epsilon}{r h}.$$

\end{lemma}

\begin{proof}
Consider the following equations
$$\left\{
  \begin{array}{ll}
    (A, \omega) \left(
                   \begin{array}{c}
                     I \\
                     t \\
                   \end{array}
                 \right) + \varepsilon P_{01} = 0
 \\
    \langle \omega, I\rangle+ \frac{1}{2} \langle I, AI\rangle+ [R](I) = 0
  \end{array}
\right.,
$$
which can be rewritten as
\begin{eqnarray}\label{Eq10}
\left(
     \begin{array}{cc}
       A & \omega \\
       \omega^T & 0 \\
     \end{array}
   \right) \left(
             \begin{array}{c}
               I \\
               t \\
             \end{array}
           \right)+ \left(
                      \begin{array}{c}
                        \varepsilon P_{01} \\
                        \frac{1}{2}\langle I, AI\rangle +[R](I) \\
                      \end{array}
                    \right) =0.
\end{eqnarray}
Combining conditions $\bf{(K')}$ and $\bf{(I')}$,  there exists an orthogonal matrix $O_1$ such that
$$O_1^T\left(
     \begin{array}{cc}
       A & \omega \\
       \omega^T & 0 \\
     \end{array}
   \right)O_1 = \left(
                  \begin{array}{cc}
                    \hat{A} & \hat{B} \\
                    \hat{C} & \hat{D} \\
                  \end{array}
                \right),
$$ with $rank (\hat{A}, \hat{B}) = n_1 + 1.$ Moreover, $\hat{A}$ is nonsingular.
Given that $rank A = n_1$, denote the first $n_1+1$ and the last $n - n_1$ components of $O_1^T\left(
                                                                                                \begin{array}{c}
                                                                                                  I \\
                                                                                                  t \\
                                                                                                \end{array}
                                                                                              \right)
$ and $O_1^T \left(
                      \begin{array}{c}
                        \varepsilon P_{01} \\
                        \frac{1}{2}\langle I, AI\rangle +[R](I) \\
                      \end{array}
                    \right)$ by $I_1 = \left(
                                         \begin{array}{c}
                                           \tilde{I} \\
                                           t \\
                                         \end{array}
                                       \right)
$, $\bar{I}$, $\breve{P} = \left(
                       \begin{array}{c}
                         P_{01} \\
                          \frac{1}{2}\langle I, AI\rangle +[R](I) \\
                       \end{array}
                     \right)
$ and $\bar{P}_{01}$, respectively. Similarly, there exists a matrix $$T_2  = \left(
                        \begin{array}{cc}
                          I & 0 \\
                          D_2 & I \\
                        \end{array}
                      \right)$$
such that (\ref{Eq10}) can be transformed to
\begin{eqnarray}\label{Eq12}
\left(
                                                                                     \begin{array}{cc}
                                                                                      \hat{A} & \hat{B} \\
                                                                                       0 & 0 \\
                                                                                     \end{array}
                                                                                   \right) \left(
                                                                                             \begin{array}{c}
                                                                                               I_1 \\
                                                                                               \bar{I} \\
                                                                                             \end{array}
                                                                                           \right)  + \varepsilon\left(
                                                                                                             \begin{array}{c}
                                                                                                                \breve{P }\\
                                                                                                               D_2 \breve{P }+ \bar{P}_{01} \\
                                                                                                             \end{array}
                                                                                                           \right)   = 0
                                                                                     .
\end{eqnarray}
Redefine $I_1^*  =  \left(
                                                                             \begin{array}{c}
                                                                               I_1 \\
                                                                               0 \\
                                                                             \end{array}
                                                                           \right)
$. Clearly, $I_1^*$ is a special solution of the following equation:
\begin{eqnarray}\label{Eq11-1}
\left(
                                                                                     \begin{array}{cc}
                                                                                     \hat{ A} & \hat{B} \\
                                                                                       0 & 0 \\
                                                                                     \end{array}
                                                                                   \right) \left(
                                                                                             \begin{array}{c}
                                                                                               I_1 \\
                                                                                               \bar{I} \\
                                                                                             \end{array}
                                                                                           \right)  + \varepsilon\left(
                                                                                                             \begin{array}{c}
                                                                                                                \breve{P }\\
                                                                                                               0 \\
                                                                                                             \end{array}
                                                                                                           \right)   = 0
                                                                                     .
\end{eqnarray}
Using the symplectic transformation $\phi: \theta\rightarrow \theta,~I\rightarrow I+ I_1^*,$ we have
\begin{eqnarray*}
\mathcal{H}_+&=& \bar{\mathcal{H}}_+ \circ \phi  = e_+ + \langle\omega_+, y\rangle+ {h}_+ + \mathcal{P}_+,
\end{eqnarray*}
where
\begin{eqnarray*}
e_+&=& e, \omega_+ = (1-t)\omega + \left(
                        \begin{array}{c}
                          0 \\
                         D_2 \breve{P}+ \bar{P}_{01} \\
                        \end{array}
                      \right)
,\\
A_+ &=& A + \varepsilon \partial_I^2 [\mathcal{R}](I_1^*), h_+ = \langle I, A_+ I\rangle, \mathcal{P}_+= \bar{\mathcal{P}}_+\circ \phi.
\end{eqnarray*}
The proof that $\phi$ is an identical map is similar to \textbf{Lemma~\ref{LM2}} and we omit the details.

\end{proof}

\subsubsection{Coordinate Transformation}
\begin{lemma}Assume
\begin{itemize}
  \item [$(e)$] $\epsilon< \gamma \eta^m r^2 \sigma^{\tau+1}.$
\end{itemize}
Then symplectic transformation $\phi_F^1$ is nearly identity. Moreover,
\begin{eqnarray*}
|U - id |_{\frac{r}{2}, s- 3\sigma} &<& \frac{\epsilon}{ \gamma \sigma^{n+1}},  |V- id|_{\frac{r}{2}, s- 3\sigma} < \frac{\epsilon}{\gamma r \sigma^\tau},\\
|U_I - I|_{\eta r, s- 3\sigma} &<& \frac{\epsilon}{ \gamma r \sigma^{\tau+1}}, |V_\theta - I|_{\frac{r}{2}, s- 4\sigma} < \frac{\epsilon}{\gamma r \sigma^{\tau+1}}, \\
|U_\theta|_{\frac{r}{2}, s- 4\sigma} &<& \frac{\epsilon}{ \gamma \sigma^{\nu+1}}, |V_I|_{\eta r, s- 3\sigma}< \frac{\epsilon}{ \gamma r^2 \sigma^{\tau+1}}.
\end{eqnarray*}
\end{lemma}

\begin{proof}
Using Cauchy estimate, we have
\begin{eqnarray*}
|F_\theta|_{r, s- 3\sigma} &\leq& \frac{\epsilon}{ \gamma \sigma^{\tau+1}},\\
|F_I|_{\frac{r}{2}, s- 2\sigma} &\leq& \frac{\epsilon}{ \gamma r \sigma^{\tau}}.
\end{eqnarray*}
Thus, we obtain
 \begin{eqnarray*}
 \frac{1}{r} |F_\theta|, \frac{1}{\sigma} |F_I| < \frac{\epsilon}{\gamma r \sigma^{\tau+1}}
 \end{eqnarray*}
uniformly on $D_{\frac{r}{2}, s- 3\sigma}$.

According to assumption $(e)$, we have
\begin{eqnarray*}
|F_\theta|_{\frac{r}{2}, s-3\sigma} \leq \frac{\epsilon}{\gamma \sigma^{\tau+1}} \leq \frac{\gamma \eta^m r^2 \sigma^{\tau+1}}{\gamma \sigma^{\tau+1}} < \eta r\leq \frac{r}{8},\\
|F_I|_{\frac{r}{2}, s-3\sigma} \leq \frac{\epsilon}{\gamma r \sigma^\tau} \leq \frac{\gamma \eta^m r^2 \sigma^{\tau+1}}{\gamma r \sigma^\tau} < \eta^m \sigma\leq \sigma,
\end{eqnarray*}
uniformly in $\xi$. Therefore, the time-1-map is well defined on $D_{\frac{r}{4}, s- 4\sigma}$, with
\begin{eqnarray}\label{time-1-map}
\phi_F^1: D_{\frac{r}{4}, s-4 \sigma} \rightarrow D_{\frac{r}{2}, s- 3\sigma}
\end{eqnarray}
and
\begin{eqnarray*}
|U - id| \leq |F_\theta| < \frac{\epsilon}{\gamma \sigma^{\tau+1}}, |V - id| \leq |F_I| < \frac{\epsilon}{\gamma r\sigma^\tau}
\end{eqnarray*}
on that domain for $\Phi = (U, V)$. The Jacobian of $\Phi$ is $D\Phi = \left(
                                                                         \begin{array}{cc}
                                                                           U_I & U_\theta \\
                                                                           V_I & V_\theta \\
                                                                         \end{array}
                                                                       \right).
$ By the preceding estimates and Cauchy estimate, we find
\begin{eqnarray*}
|U_I - Id| < \frac{\epsilon}{ \gamma r\sigma^{\tau+1}}, |V_I - Id| < \frac{\epsilon}{ \gamma r\sigma^{\tau+1}},|U_\theta| < \frac{\epsilon}{\gamma \sigma^{\tau+2}}, |V_I| < \frac{\epsilon}{\gamma r^2 \sigma^\tau}.
\end{eqnarray*}
\end{proof}

\subsubsection{New Error Term}
\begin{lemma}
Let $\sigma_+ = \frac{\sigma}{2}$, $r_+ = r \eta$, $\eta_+ = \eta^{\frac{2m -2 }{m}}$. Then
\begin{eqnarray*}
|\mathcal{P}_+|_{\eta r, s- 5\sigma, \eta h} \leq \min\limits_{0\leq i\leq m}\{\varepsilon_i\} \gamma r_+^2 \eta_+^m \sigma_+^{\tau+1}.
\end{eqnarray*}
\end{lemma}

\begin{proof}
Directly,
\begin{eqnarray*}
|\{R, F\}|_{\frac{r}{2}, s- 3\sigma} &<& |R_I||F_\theta| - |R_\theta||F_I|\\
&<& \frac{\varepsilon}{r} \frac{\varepsilon}{\tilde{\varepsilon}\gamma \sigma^\tau} - \frac{\varepsilon}{\sigma} \frac{\varepsilon}{\tilde{\varepsilon}\gamma r \sigma^\tau}<\frac{\varepsilon \epsilon}{\gamma r \sigma^{\tau+1}},\\
|\{[R], F\}|_{\frac{r}{2}, s- 3\sigma}&<& \frac{\varepsilon \epsilon}{ \gamma r \sigma^{\tau+1}}.
\end{eqnarray*}

Together with (\ref{time-1-map}) and $\eta< \frac{1}{8},$ we get
\begin{eqnarray*}
|\int_0^1 \{(1-t)[R]+ t R, F\}\circ X_F^t dt|_{\eta r, s- 5\sigma} &\leq& |\{(1-t)[R]+ t R, F\}|_{\frac{r}{2}, s- 4\sigma}\\
&<& \frac{\varepsilon \epsilon}{\gamma r \sigma^{\tau+1}},\\
|(P - R)\circ \Phi|_{\eta r, s- 5\sigma} &\leq& |P- R|_{2\eta r, s- 4\sigma}\\
&\leq& |P- Q|_{2\eta r, s- 4\sigma} + |Q- R|_{2\eta r, s- 4\sigma}\\
&\leq& \varepsilon \eta^m.
\end{eqnarray*}
Hence
\begin{eqnarray*}
|\mathcal{P}_+|_{\eta r, s- 5\sigma, \eta h} &<& \frac{\varepsilon^2}{\gamma r^2 \sigma^{\tau+1}} + \varepsilon \eta^m\\
&\leq& \varepsilon \eta^m \leq \min\limits_{0\leq i\leq m}\{\varepsilon_i\} \gamma r_+^2 \eta_+^m \sigma_+^{\tau+1}.
\end{eqnarray*}
\end{proof}

\subsection{KAM Iteration}\label{KAMiteration}
For $\nu = 0,1,\cdots$, let
\begin{eqnarray*}
r_{\nu} &=& r_{\nu-1} \eta_{\nu-1}, \eta_{\nu} = \eta_{\nu-1}^{\frac{2m -2 }{m}}, K_\nu= ([\log \frac{1}{\eta_\nu^m}] +1)^{a},\\
s_{\nu+1} &=& s_\nu - 5 \sigma_\nu, \sigma_\nu = \frac{3}{20} s_\nu, h_\nu = h_{\nu-1} \eta_{\nu-1}.
\end{eqnarray*}

\begin{lemma}\label{Iteration}
The KAM step described above is valid for all  $\nu = 0, 1, \cdots,$ and the following holds for all $\nu = 1,2,\cdots.$
There exists a transformation:
\begin{eqnarray*}
\mathcal{F}_{\nu} = \Phi_{\nu}\circ \phi_{\nu} = (U_{\nu}, V_{\nu}, \phi_{\nu}):  D_{s_{\nu+1}, r_{\nu+1}} \times \Lambda_{h_{\nu+1}} \rightarrow D_{s_\nu, r_\nu}  \times \Lambda_{h_{\nu}} ,
\end{eqnarray*}
such that
\begin{eqnarray*}
\mathcal{H}\circ \mathcal{F} = \mathcal{N}_{\nu+1} + \mathcal{P}_{\nu+1},
\end{eqnarray*}
where
\begin{eqnarray*}
\mathcal{N}_{\nu+1}(\xi, \varepsilon_i) = e_{\nu+1}(\xi, \varepsilon_i) + \langle \omega(\xi, \varepsilon_i), I\rangle + \frac{1}{2} \langle I, A_{\nu+1}(\xi, \varepsilon_i)I \rangle + \sum\limits_{3\leq |\jmath| \leq m-1} h_{\jmath} (\xi, \varepsilon_i) I^\jmath ,\\
|\mathcal{P}_{\nu+1}|_{r_{\nu+1}, s_{\nu+1}, h_{\nu+1}} < \varepsilon_{\nu+1}< \tilde{\varepsilon}\gamma r_{\nu+1}^2 \eta_{\nu+1}^m \sigma_{\nu+1}^{\tau + 1}.
\end{eqnarray*}
Moreover, we have
\begin{eqnarray*}
|U_{\nu} - id |_{\frac{r_{\nu}}{2}, s_{\nu}- 3\sigma}&<& \frac{\epsilon_{\nu}}{ \gamma_{\nu} \sigma_{\nu}^{n+1}},  |V_{\nu}- id|_{\frac{r_{\nu}}{2}, s_{\nu}- 3\sigma_{\nu}} < \frac{\epsilon_{\nu}}{\gamma_{\nu} r_{\nu} \sigma_{\nu}^\tau},\\
|{U_{\nu}}_I - I|_{\eta_{\nu} r_{\nu}, s_{\nu}- 3\sigma_{\nu}} &<& \frac{\epsilon_{\nu}}{ \gamma_{\nu} r_{\nu} \sigma_{\nu}^{\tau+1}},  |{V_{\nu}}_\theta - I|_{\frac{r_{\nu}}{2}, s_{\nu}- 4\sigma_{\nu}} < \frac{\epsilon_{\nu}}{ \gamma_{\nu} r_{\nu} \sigma_{\nu}^{\tau+1}},\\
|{U_{\nu}}_\theta|_{\frac{r_{\nu}}{2}, s_{\nu}- 4\sigma_{\nu}} &<& \frac{\epsilon_{\nu}}{\gamma_{\nu} \sigma_{\nu}^{\nu+1}}, |{V_{\nu}}_I|_{\eta_{\nu} r_{\nu}, s_{\nu}- 3\sigma_{\nu}}< \frac{\epsilon_{\nu}}{\gamma_{\nu} r_{\nu}^2 \sigma_{\nu}^{\tau+1}},\\
 |\phi_{\nu} - id|_{\frac{h_{\nu}}{2}} &<& \frac{\varepsilon_{\nu}}{r_{\nu}}, |D\phi_{\nu} - I|_{\eta_{\nu} h_{\nu}}< \frac{\varepsilon_{\nu}}{r_{\nu}},\\
 |\partial_\xi(U_{\nu} - id )|_{\frac{r_{\nu}}{2}, s_{\nu}- 3\sigma_{\nu}, \eta_{\nu} h_{\nu}} &<& \frac{\epsilon_{\nu}}{ \gamma_{\nu} h_{\nu} \sigma_{\nu}^{n+1}},
 |\partial_\xi (V_{\nu}- id)|_{\frac{r_{\nu}}{2}, s_{\nu}- 3\sigma_{\nu}, \eta_{\nu} h_{\nu}} < \frac{\epsilon_{\nu}}{ \gamma_{\nu} r_{\nu} h_{\nu} \sigma_{\nu}^\tau}.
\end{eqnarray*}

\end{lemma}
\begin{proof}
The proof is equivalent to check assumptions $(a)- (e)$. Directly,
\begin{eqnarray*}
\eta_\nu &=&  \eta_{\nu-1}^{\frac{2m-2}{m}}\\
 &=& \eta_0^{(\frac{2m-2}{m})^\nu},\\
 r_\nu &=& r_{\nu-1} \eta_{\nu-1}\\
 &=& r_0 \eta_0 \eta_1\cdots \eta_{\nu-1}\\
 &=& r_0 \eta_0 \eta_0^{\frac{2m-2}{m}}\cdots \eta_0^{(\frac{2m-2}{m})^{\nu-1}}\\
 &=& r_0 \eta_0^{\frac{m^\nu - (2m -2)^\nu}{m^{\nu-1} (2-m)}}.
\end{eqnarray*}
Then
\begin{eqnarray*}
\eta_\nu^m  &=& \eta_0^{(\frac{2m-2}{m})^\nu m},\\
~[\log \frac{1}{\eta_\nu^m}] &=& (\frac{2m-2}{m})^\nu m \log \frac{1}{\eta_0}.
\end{eqnarray*}
Let $a\geq \log _{2- \frac{2}{m}}^4,$ where $m>2.$ For $a>1$ and $b>0$, $\eta_0^a \log \frac{1}{\eta_0^b} \rightarrow 0$ as $\eta_0 \rightarrow 0.$
Therefore,
\begin{eqnarray*}
r_\nu K_\nu^{\tau+1} &=& r_0 \eta_0^{\frac{m^\nu - (2m -2)^\nu} {m^{\nu-1} (2-m)}}( [(\frac{2m-2}{m})^\nu m \log \frac{1}{\eta_0}]+1)^{(\tau+1)a}\\
&\leq& r_0 \eta_0^{\frac{ m^\nu - (2m-2)^\nu}{m^{\nu-1} (2-m)}} \big((\frac{2m-2}{m})^{(\tau+1)a}\big)^\nu (2m)^{a(\tau+1)} (\log \frac{1}{\eta_0})^{(\tau+1)a}\\
&<&1,
\end{eqnarray*}
which implies assumption $(d)$ hold. Let $h_0 = r_0$. Then assumption $(c)$ holds.

Direct,
\begin{eqnarray*}
&~&\int_{K_{\nu+1}}^{\infty} \lambda^n e^{-\lambda \sigma} d \lambda \\
&=& \lambda^n (-\frac{1}{\sigma}) e^{-\lambda \sigma}|_{K_{\nu+1}}^{\infty} - \int_{K_{\nu+1}}^{\infty} (-\frac{1}{\sigma}) e^{- \lambda \sigma} n \lambda^{n-1} d \lambda\\
&=& \lambda^n (-\frac{1}{\sigma}) e^{-\lambda \sigma}|_{K_{\nu+1}}^{\infty} + \int_{K_{\nu+1}}^{\infty} \frac{n}{\sigma} e^{- \lambda \sigma} \lambda^{n-1} d \lambda\\
&=& \lambda^n (-\frac{1}{\sigma}) e^{-\lambda \sigma}|_{K_{\nu+1}}^{\infty} + \int_{K_{\nu+1}}^{\infty} \frac{n}{\sigma} \lambda^{n-1} d (\frac{e^{- \lambda \sigma}}{-\sigma})\\
&=& -\lambda^n \frac{1}{\sigma} e^{-\lambda \sigma}|_{K_{\nu+1}}^{\infty} - \frac{n}{\sigma^2} \lambda^{n-1} e^{-\lambda \sigma} |_{K_{\nu+1}}^{\infty} + \int_{K_{\nu+1}}^{\infty} \frac{n(n-1)}{\sigma^2} \lambda^{n-2} d(-\frac{e^{\lambda \sigma}}{\sigma})\\
&=& \frac{K_{\nu+1}^n}{\sigma} e^{-K_{\nu+1} \sigma} + \frac{n K_{\nu+1}^{n-1}}{\sigma^2} e^{- K_{\nu+1} \sigma} +  \frac{n (n-1) K_{\nu+1}^{n-2}}{\sigma^3} e^{- K_{\nu+1} \sigma} + \cdots \\
&~&+  \frac{n (n-1)\cdots 2 }{\sigma^n} K_{\nu+1} e^{- K_{\nu+1} \sigma} + \frac{n!}{\sigma^{n+1}} e^{- K_{\nu+1} \sigma}\\
&<& \frac{n! K_{\nu+1}^n e^{- K_{\nu+1} \sigma}}{\sigma^{n+1}}.
\end{eqnarray*}
Since
\begin{eqnarray*}
s_\nu &=& \frac{1}{4} s_{\nu-1}= \frac{1}{4^{\nu}} s_0,\\
\sigma_\nu &=& \frac{3}{20} s_\nu = \frac{3}{20}\frac{1}{4^{\nu}} s_0,
\end{eqnarray*}
we have
\begin{eqnarray*}
\sigma_\nu K_{\nu} &=&\frac{3}{20}\frac{1}{4^{\nu}} s_0 ( [(\frac{2m-2}{m})^\nu m \log \frac{1}{\eta_0}]+1)^{a}\\
&=& \frac{3}{20} (\frac{1}{4})^{\nu} s_0 (2m)^a (\log \frac{1}{\eta_0})^a (2-\frac{2}{m})^{\nu a}\\
&=& \frac{3}{20} (\frac{1}{4} (2- \frac{2}{m})^a)^\nu s_0 (\log \frac{1}{\eta_0})^a  (2m)^a.
\end{eqnarray*}
Since
\begin{eqnarray*}
\frac{1}{4} (2-\frac{2}{m})^a > 1
&\Leftrightarrow&
(2-\frac{2}{m})^a > 4,\\
&\Leftrightarrow&
a > \frac{\log 4}{\log (2 - \frac{2}{m})},
\end{eqnarray*}
 for given $a > \frac{\log 4}{\log (2 - \frac{2}{m})}$, we have $\sigma_\nu K_{\nu} > n+1$, i.e., assumption $(a)$ holds.

Further,
\begin{eqnarray*}
&~&\log n! + \log K_{\nu+1}^n - \log \sigma^{n+1} + \log e^{- K_{\nu+1} \sigma}\\
&=&\log n! + n \log K_{\nu+1} - (n+1)\log \sigma - K_{\nu+1} \sigma\\
&\leq& c (n) \log K_{\nu+1} - c(n) K_{\nu+1} \sigma\\
&\leq& c (n) \log ( [(\frac{2m-2}{m})^\nu m \log \frac{1}{\eta_0}]+1)^{a}\\
&~&- c(n)\frac{3}{20} (\frac{1}{4} (2- \frac{2}{m})^a)^\nu s_0 (\log \frac{1}{\eta_0})^a  (2m)^a\\
&\leq& - (\frac{2m-2}{m})^\nu m  \log \frac{1}{\eta_0},
\end{eqnarray*}
i.e., assumption $(b)$ holds. According to the definitions of $\epsilon_\nu$, $\varepsilon_\nu$, $\gamma_\nu$, $r_\nu,$ $\eta_\nu$, $\sigma_\nu$ and $h_\nu$, assumptions $(e)$ holds.
\end{proof}

\subsection{Convergence}\label{convergence}
Let $\Phi_j = \phi_{F_j}^1 \circ \phi_j$ and $\hat{\Phi}_\nu = \Phi_1\circ\cdots\circ \Phi_\nu$. Then $\hat{\Phi}_\nu: D_{\nu} \times \Lambda_{\nu} \rightarrow D_{0} \times \Lambda_{0}$. Moreover, we have $|\tilde{W}_0 (\hat{\Phi}_\nu - \hat{\Phi}_{\nu-1})|_{D_{\nu+1}\times \Lambda_{\nu+1}} \leq \max \{\frac{\epsilon}{ \gamma r_{\nu+1} \sigma_{\nu+1}^{\tau+1} }, \frac{\varepsilon_{\nu+1}}{r_{\nu+1} h_{\nu+1}}\}$, where $\tilde{W}_0 = diag (r_0^{-1}I, \sigma_0^{-1}I, h_0^{-1}I).$ In fact, during the KAM step, we have
\begin{eqnarray*}
|\tilde{W}_\nu (\Phi_\nu - id)|_{D_{\nu+1} \times \Lambda_{\nu+1}}, |\tilde{W}_\nu (\bar{D}\Phi_\nu - Id) \tilde{W}_\nu^{-1}|_{D_{\nu+1} \times \Lambda_{\nu+1}} < \max (\frac{ \epsilon}{\gamma r_\nu \sigma_\nu^{\tau+1}}, \frac{\varepsilon_\nu}{r_\nu h_\nu}),
\end{eqnarray*}
where $\bar{D}$ denotes the Jacobian with respect to $I,$ $\theta$ and $\xi$, and $\tilde{W}_j = diag (r_\nu^{-1}Id, \sigma_\nu^{-1}Id, h_\nu^{-1} Id)$. Then
\begin{eqnarray*}
|\tilde{W}_0( \hat{\Phi}_{\nu+1} - \hat{\Phi}_{\nu})| &=& |\tilde{W}_0( \hat{\Phi}_{\nu} \circ \Phi_{\nu}  - \hat{\Phi}_{\nu}) |\\
 &\leq& |\tilde{W}_0 \bar{D} \hat{\Phi}_{\nu} \tilde{W}_{\nu}^{-1}||\tilde{W}_{\nu}(\Phi_\nu -  id)|\\
 &<&  |\tilde{W}_\nu (\Phi_\nu - id)| < \max (\frac{ \epsilon}{\gamma r_\nu \sigma_\nu^{\tau+1}}, \frac{\varepsilon_\nu}{r_\nu h_\nu}).
\end{eqnarray*}
By induction, we have $\bar{D}\hat{\Phi}_{\nu} = \bar{D}\Phi_{0} \circ \cdots\circ \bar{D}\Phi_{\nu-1}$, with the Jacobians evaluated at different points, which we do not indicate. Since $|\tilde{W}_\nu \tilde{W}_{\nu+1}^{-1}|\leq\max (\frac{r_{\nu+1}}{r_\nu}, \frac{s_{\nu+1}}{s_\nu}, \frac{h_{\nu+1}}{h_\nu}) \leq1$ for all $\nu$, we have
\begin{eqnarray*}
|\tilde{W}_0 \bar{D}\hat{\Phi}_\nu \tilde{W}_\nu^{-1}| &\leq& |\tilde{W}_0 \bar{D}{\Phi}_0\circ \cdots\circ{\Phi}_{\nu-1} \tilde{W}_\nu^{-1}|\\
&\leq& |\tilde{W}_0 \bar{D}{\Phi}_0\tilde{W}_0^{-1}| |\tilde{W}_0 \tilde{W}_1^{-1}|\cdots|\tilde{W}_{j-1} \bar{D}\Phi_{\nu-1} \tilde{W}_{j-1}^{-1}| |\tilde{W}_{j-1}\tilde{W}_{j}^{-1}|\\
&\leq& \prod\limits_{\nu} \big(1+ c_1\max (\frac{ \epsilon}{\gamma r_\nu \sigma_\nu^{\tau+1}}, \frac{\varepsilon_\nu}{r_\nu h_\nu})\big)\prod\limits_{\nu} \max (\frac{r_{\nu+1}}{r_\nu}, \frac{s_{\nu+1}}{s_\nu}, \frac{h_{\nu+1}}{h_\nu}).
\end{eqnarray*}
This is uniformly bounded and small, as $\max (\frac{ \epsilon}{\gamma r_\nu \sigma_\nu^{\tau+1}}, \frac{\varepsilon_\nu}{r_\nu h_\nu})$ converges rapidly to zero.

The map $\hat{\Phi}_{\nu}$ converge uniformly on
\begin{eqnarray*}
\bigcap\limits_{\nu\geq 0} D_\nu \times \Lambda_\nu = T_*\times \Omega_{\gamma,\tau}^{\gamma}, T_* = \{\xi\}\times \mathds{T}^n,
\end{eqnarray*}
to a map $\hat{\Phi}$, which consists of a family of embedding $\Psi: \mathds{T}^n \times \Omega_{\gamma,\tau}^{\gamma} \rightarrow \Lambda \times \mathds{T}^n$ and a parameter transformation $\phi: \Omega_{\gamma,\tau}^{\gamma} \rightarrow \Lambda$. Both $\Psi$ and $\phi$ are real analytic on $\mathds{T}^n$ and uniformly continuous on $\Omega_{\gamma,\tau}^\gamma$. Moreover, we have
\begin{eqnarray*}
|\tilde{W}_0 (\tilde{\Phi} -  id)| < \max (\frac{\epsilon_{0}}{ \gamma r_{0} \sigma_{0}^{\tau+1} }, \frac{\varepsilon_{0}}{r_{0} h_{0}}),
\end{eqnarray*}
uniformly on $T_*\times \Omega_{\gamma,\tau}^\gamma$. Let $\hat{\phi}_\nu = \phi_{F_1}^1\circ\cdots\circ \phi_{F_\nu}^1$. From the estimate $|H\circ \hat{\Phi}_{\nu} - \mathcal{N}_{\nu}|_{D_\nu\times \Lambda_\nu}$,  we obtain
\begin{eqnarray*}
|\hat{W}_\nu( J (D \hat{\phi}_\nu )^t \nabla H\circ \hat{\Phi}_\nu - J \nabla \mathcal{N}_\nu )| < \max (\frac{ \epsilon_\nu}{ \gamma r_\nu \sigma_\nu^{\tau+1}}, \frac{\varepsilon_\nu}{r_\nu h_\nu})
\end{eqnarray*}
on $T_*\times \Lambda_\nu$ with $\hat{W}_\nu =diag (r_\nu^{-1} Id, \sigma_\nu^{-1} Id)$. The symplectic nature of the map $\hat{\phi}_\nu$ and the uniform estimate of ${\hat{W}}_0 \bar{D}\hat{\Phi}_\nu {\hat{W}}_\nu^{-1}$ imply
\begin{eqnarray*}
|X_H\circ \hat{\Phi}_\nu - D \hat{\phi}_\nu \cdot X_N|< \frac{\varepsilon_\nu}{r_\nu \sigma_\nu}
\end{eqnarray*}
on $T_* \times \Omega_{\gamma, \tau}^{\gamma}$ for all $\nu$, where $X_N$ is the Hamiltonian vector field of $N$. Going to the limit, we obtain
\begin{eqnarray*}
X_H \circ \hat{\Phi} = D \phi \cdot X_N.
\end{eqnarray*}
Thus, $\Psi$ is an embedding of the Kronecker torus $(\mathds{T}^n, \omega(\xi))$ as an invariant torus of the Hamiltonian vector field $X_H$ at the parameter $\phi(\xi)$.

\subsection{Measure Estimate}\label{measureestimate}
\begin{lemma}\label{estimate measure}
 Let $\tau > n-1.$ Denote $R_{k, \nu+1} = \{\xi \in \Lambda_\nu: |\langle k, \omega\rangle|< \frac{\min\limits_{0\leq j\leq m}\{\varepsilon_j\}\gamma_\nu}{|k|^{\tau+1}}, K_\nu< |k| \leq K_{\nu+1}\}.$ Assume ${(R)}$ hold. Then $\bigcup\limits_{\nu = 0}^\infty \bigcup\limits_{K_\nu < |k| \leq K_{\nu+1}} R_{k, \nu+1} = O(\gamma_0^{\frac{1}{N+1}})$ as $\gamma_0 \rightarrow 0.$
\end{lemma}

\begin{proof}
Let $L_{k,\nu} = |k|\langle \varsigma, \omega_\nu(\lambda)\rangle$, where $\varsigma = \frac{k}{|k|}\in S^n$ and $S^n$ is a $n$-dimensional ball. According to the Taylor series, for a given $\xi\in \Lambda_\nu$,
\begin{eqnarray*}
{L}_{k, \nu}= |k|\varsigma^T \Omega_\nu (\xi) \tilde{\xi},
\end{eqnarray*}
where $\Omega_\nu (\xi) = \big( \omega_\nu(\xi), \cdots, \partial_\lambda^\alpha \omega_\nu(\xi), \int_0^1 (1-t)^{|\alpha+1|} \partial_\lambda^ {\alpha+1}\omega_\nu(\xi+ t \xi_1) \big)$, $\hat{\xi} = \xi_1 - \xi = (\hat{\xi}_1, \cdots, \hat{\xi}_n),$ $\tilde{\xi}=(1, \hat{\xi}, \cdots, \hat{\xi}^\alpha, \hat{\xi}^{\alpha+1})$.

Let $Q_{\xi, \nu}= (q_{ij})_{\breve{\iota}\times \breve{\iota}}$ be a matrix such that $\Omega_\nu(\xi) Q_{\xi, \nu}$ implies that only some columns of $\Omega_\nu(\xi)$ are changed. According to condition $(R)$, we have $rank \Omega_\nu(\xi) = n$ for $\xi\in \Lambda_\nu\subset \Lambda$, meaning there exists a matrix $Q_{\xi,\nu}= (q_{ij})_{\breve{\iota}\times \breve{\iota}}$ such that $\Omega_\nu(\xi)Q_{\xi,\nu} = \big(A_\nu (\xi), B_\nu(\xi)\big)$, where $A_\nu(\xi) = (a_{ij})_{n\times n}$ is nonsingular. Denote $\Lambda_{\xi,\nu}$ as the neighborhood of $\xi$ and $\bar{\Lambda}_{\xi,\nu}$ the closure of $\Lambda_{\xi,\nu}$. Thus, $\det A_\nu(\xi) \neq 0$ for $\xi \in \bar{\Lambda}_{\xi, \nu}$. Consequently, there exists an orthogonal matrix $Q_{\xi,\nu}$ such that $\Omega_\nu(\xi) Q_{\xi, \nu} = (A_{\nu}(\xi), B_{\nu}(\xi))$ for $\xi\in \bar{\Lambda}_{\xi, \nu}$, where $\det A_{\nu}(\xi) \neq 0$ on $\bar{\Lambda}_{\xi, \nu}.$

Let $\check{\xi}_{1, \nu}\leq \cdots \leq\check{\xi}_{n,\nu}$ be the eigenvalues of $(A_{\nu}(\xi)A_{\nu}^* (\xi)+ B_{\nu}(\xi)B_{\nu}^* (\xi))$. Since $rank (A_{\nu}(\xi)A_{\nu}^* (\xi)+ B_{\nu}(\xi)B_{\nu}^* (\xi)) = rank (A_{\nu}(\xi), B_{\nu}(\xi))$ (\cite{Horn}), there is a unitary $U_{\nu}$ and a real diagonal $V_{\nu} = diag(\check{\xi}_{1,\nu}, \cdots, \check{\xi}_{n,\nu})$ such that $(A_{\nu}(\xi)A_{\nu}^* (\xi)+ B_{\nu}(\xi)B_{\nu}^* (\xi)) = U_{\nu} V_{\nu} U_{\nu}^*$. Therefore, using Poincar\'{e} separation theorem and Lemma \ref{Eigenvalue} in Appendix \ref{Inverse},
\begin{eqnarray*}
\varsigma^* (A_{\nu}(\xi)A_{\nu}^* (\xi)+ B_{\nu}(\xi)B_{\nu}^* (\xi)) \varsigma &=& \varsigma^* (A_{\nu}(\xi)A_{\nu}^* (\xi)+ B_{\nu}(\xi)B_{\nu}^* (\xi)) \varsigma\\
 &=& \varsigma^*U_{\nu}  diag(\check{\xi}_{1,\nu}, \cdots, \check{\xi}_{n,\nu}) U_{\nu}^* \varsigma\\
 &\geq& \varsigma^*U_{\nu}  \check{\xi}_{1,\nu} I_n U_{\nu}^*\varsigma\\
 &\geq& \check{\xi}_{1, \nu} \geq c\min\limits_{0\leq j\leq m}\{\varepsilon_j^2\},
\end{eqnarray*}
where $c$ is positive constant depending on the elements of $\Omega_\nu(\xi)$.

Since the nonzero eigenvalues of $\left(
                                    \begin{array}{cc}
                                      A_{\nu}^T (\xi) \varsigma\varsigma^T A_{\nu}(\xi) & A_{\nu}^T (\xi) \varsigma\varsigma^T B_{\nu}(\xi) \\
                                      B_{\nu}^T (\xi) \varsigma\varsigma^TA_{\nu}(\xi) & B_{\nu}^T(\xi) \varsigma\varsigma^T B_{\nu}(\xi) \\
                                    \end{array}
                                  \right)
$ and $\varsigma^T (A_{\nu}(\xi)A_{\nu}^T (\xi)+ B_{\nu}(\xi)B_{\nu}^T (\xi))\varsigma$ are the same, there exists an unitary matrix $\mathcal{U}_{\nu}(\xi)$ such that
\begin{eqnarray*}
\left(
                                    \begin{array}{cc}
                                      A_{\nu}^T (\xi) \varsigma\varsigma^T A_{\nu}(\xi) & A_{\nu}^T (\xi) \varsigma\varsigma^T B_{\nu}(\xi) \\
                                      B_{\nu}^T (\xi) \varsigma\varsigma^TA_{\nu}(\xi) & B_{\nu}^T(\xi) \varsigma\varsigma^T B_{\nu}(\xi) \\
                                    \end{array}
                                  \right) = \mathcal{U}_{\nu}(\xi) diag (0, \cdots,0, \check{\xi}_\nu) \mathcal{U}_{\nu}^*(\xi),
\end{eqnarray*}
where $\check{\xi}_\nu = \varsigma^*U_{\nu}  diag(\check{\xi}_{1,\nu}, \cdots, \check{\xi}_{n,\nu}) U_{\nu}^* \varsigma.$ Let $a$ be the dimension of $\mathcal{U}_{\nu}$ and $\mathcal{U}_\nu = (u_{j_1i, \nu})_{a\times a}$. Using Hadamard's inequality (\cite{Horn}), $\det \mathcal{U}_{\nu}^* \mathcal{U}_{\nu} = \det I \leq \sum\limits_{0\leq j_1\leq a}u_{j_1i, \nu}^2$, which implies that there exists an $i$ such that $\sum\limits_{0\leq j_1\leq a}u_{j_1i, \nu}^2\geq \frac{1}{2}$.

 Denote $(\mathcal{U}_{\nu}^* Q_{\xi,\nu} )_i$ as the $i-$th row of $\mathcal{U}_{\nu}^* Q_{\xi,\nu}$. Then we have
\begin{eqnarray*}
||(\mathcal{U}_{\nu}^* Q_{\xi,\nu}^{-1} \tilde{\xi})_i||_2 = ||(\mathcal{U}_{\nu}^* Q_{\xi, \nu}^{-1} )_i \tilde{\xi} ||_2 \geq \sum\limits_{0\leq j_1\leq a}u_{j_1i,\nu}^2(\min\limits_{1\leq j\leq n} |\hat{\xi}_j|)^{2N+2}.
\end{eqnarray*}
Therefore,
\begin{eqnarray}
\label{Lk0} |{L}_{k, \nu}^* {L}_{k, \nu}| &=& |k|^2 |\tilde{\xi}^T Q_{\xi, \nu}Q_{\xi, \nu}^{-1} \Omega_\nu^T \varsigma\varsigma^T \Omega_\nu Q_{\xi, \nu}  Q_{\xi, \nu}^{-1} \tilde{\xi}|\\
\nonumber&=& |k|^2 |\tilde{\xi}^T Q_{\xi, \nu} \left(
                                          \begin{array}{cc}
                                            A_{\nu}^T (\xi) \varsigma \varsigma^T A_{\nu}(\xi) & A_{\nu}^T (\xi) \varsigma\varsigma^T B_{\nu}(\xi) \\
                                            B_{\nu}^T (\xi) \varsigma\varsigma^T A_{\nu}(\xi) & B_{\nu}^T(\xi) \varsigma\varsigma^T B_{\nu}(\xi) \\
                                          \end{array}
                                        \right)Q_{\xi,\nu}^{-1} \tilde{\xi}|\\
\nonumber&=&|k|^2 |\tilde{\xi}^T Q_{\xi, \nu}\mathcal{U}_{\nu} (\xi) diag (0,\cdots, 0, \check{\xi}_{\nu}) \mathcal{U}_{\nu}^*(\xi)Q_{\xi, \nu}^{-1} \tilde{\xi}|\\
\nonumber&\geq& |k|^2 \check{\xi}_{\nu} |(\mathcal{U}_{\nu}^*(\xi) Q_{\xi,\nu}^{-1} \tilde{\xi})_{i}|\\
\nonumber&\geq& |k|^2 c\min\limits_{0\leq j\leq m}\{\varepsilon_j^2\} (\min\limits_{1\leq j \leq n} |\hat{\lambda}_i|)^{2N+2}.
\end{eqnarray}
We utilize the following facts: for any given matrices $A$, $B$ and orthogonal matrix $P$, $tr AB = tr BA$ and $tr P^{-1} A P = tr A$.

Directly, we find that
\begin{eqnarray*}
~&~& |\{ \xi\in \Lambda_\nu \bigcap \bar{\Lambda}_{\xi, \nu}: |{L}_{k, \nu}^* {L}_{k, \nu}| \leq \frac{\min\limits_{0\leq j\leq m} \{\varepsilon_j^2\}\gamma_\nu^2}{|k|^{2\tau}},K_{\nu}< |k|\leq K_{\nu+1} \}| \\
&<& |\{ \xi\in \Lambda_\nu \bigcap \bar{\Lambda}_{\xi, \nu}:  c\big(\min\limits_{1\leq j\leq n}|\hat{\xi}_j|\big)^{2N + 2} \leq \frac{\gamma_\nu^2}{|k|^{2(\tau-1)}}, K_{\nu}< |k|\leq K_{\nu+1} \}|\\
 &\leq& c  D^{n-1} \frac{\gamma^{\frac{1}{N+1}}}{|k|^{\frac{\tau-1}{N+1}}},
\end{eqnarray*}
where $D$ is the exterior diameter of $\bar{\xi}_{\xi,\nu}$ with respect to the maximum norm, $n$ is the dimension of $\Lambda_\nu$, $c$ is positive constant dependent on the elements of $\Omega_\nu(\xi)$. Therefore, there exist finite sets, $\bar{\Lambda}_{\xi_i, \nu}$, $1\leq i \leq \tilde{\iota}$, such that $\mathcal{O}_\nu\subset \bigcup\limits_{i=1}^{\tilde{\iota}} \bar{\Lambda}_{\xi_i, \nu}$ and
\begin{eqnarray*}
|{L}_{k, \nu}^* (\xi){L}_{k, \nu}(\xi)| > |k|^2 c\min\limits_{0\leq j\leq m}\{\varepsilon_j^2\} \big(\min\limits_{1\leq j\leq n}|\hat{\xi}_j|\big)^{2N + 2} ~for ~\xi\in \bar{\Lambda}_{\xi_i, \nu}.
\end{eqnarray*} Hence
\begin{eqnarray*}
 |\{\xi\in \Lambda_\nu: |{L}_{k,\nu}| \leq \frac{\min\limits_{0\leq j\leq m}\{\varepsilon_j\}\gamma_\nu}{|k|^\tau}, K_{\nu}< |k|\leq K_{\nu+1}\}|< c D^{n-1} \frac{\gamma_0^{\frac{1}{N+1}}}{|k|^{\frac{\tau-1}{N+1}}},
\end{eqnarray*}
where $c$ depends on $\mathcal{O}$, $D$, $n$ and $\Omega_{\nu}$.

\end{proof}

\section{Proof of Main Theorem}\label{SEC3}
In this section, we are going to prove Theorem \ref{MainTheorem} and Theorem \ref{Theorem1}. Specifically, by combining the results from Section \ref{AKAMstep} to Section \ref{measureestimate}, we establish Theorem \ref{Theorem1}. Next, we will prove Theorem \ref{MainTheorem} using Theorem \ref{Theorem1}.

Let $\partial D$ denote the boundary of $D$. Consider $\Lambda\subseteq D$ with $dist (\partial\Lambda, \partial D) = r$. For any $\xi \in \Lambda$, define $\tilde{I} = \xi+ I$ and expand $\tilde{H}(\tilde{I})$ within the ball $B_{r}(\xi)$:
\begin{eqnarray}\label{Taylor}
\nonumber \tilde{H}(\tilde{I}, \varepsilon_i) &=& e(\xi, \varepsilon_i) + \langle \omega(\xi, \varepsilon_i), I\rangle + \frac{1}{2} \langle I, A(\xi, \varepsilon_i)I \rangle + \sum\limits_{3\leq |\jmath| \leq m-1} h_{\jmath} (\xi, \varepsilon_i) I^\jmath \\
\label{Taylor}&~&+ \sum\limits_{|\jmath| \geq m} h_{\jmath} (\xi, \varepsilon_i) I^\jmath.
\end{eqnarray}

Recall $D_{r,s} = \{(I, \theta): |I|< r, |Im \theta|< s\} \subset \mathds{C}^n \times \mathds{C}^n$ and $\Lambda_{h} = \{\xi \in \Lambda: d (\xi, \Omega_{\gamma,\tau}^\gamma) < h\} \subset \mathds{C}^n$ denote the complex neighbourhoods of the torus $\{\xi\} \times \mathds{T}^n$ and $ \Omega_{\gamma,\tau}^\gamma$, respectively, where $|\cdot|$ stands for the super-norm of real vectors. Let $r = \varepsilon^{\frac{1}{m}}$. Then we have
\begin{eqnarray*}
| \sum\limits_{|\jmath| \geq m} h_{\jmath} (\xi, \varepsilon_i) I^\jmath|_{r,s,h} < c \varepsilon.
\end{eqnarray*}
We define
 \begin{eqnarray*}
 \mathcal{N}(\xi, \varepsilon_i) &=& e(\xi, \varepsilon_i) + \langle \omega(\xi, \varepsilon_i), I\rangle + \frac{1}{2} \langle I, A(\xi, \varepsilon_i)I \rangle + \sum\limits_{3\leq |\jmath| \leq m-1} h_{\jmath} (\xi, \varepsilon_i) I^\jmath,\\
 \mathcal{P} (\xi, \varepsilon_i)&=&  \sum\limits_{|\jmath| \geq m} h_{\jmath} (\xi, \varepsilon_i) I^\jmath + \varepsilon P(I, \theta, \xi).
\end{eqnarray*}
 Moreover, $|\mathcal{P}|_{r,s,h} < c \varepsilon.$ Then, using Theorem \ref{Theorem1}, we can prove Theorem \ref{MainTheorem}.

 \begin{remark}
 There is an invariant Kronecher torus $\mathcal{T}_{\xi} = \{\xi\}\times \mathds{T}^n$ with constant vector field $X_{\mathcal{N}} = \sum\limits_{1\leq j\leq n} \omega_j(\xi) \frac{\partial}{\partial\theta_j}$ for each $\xi\in \Omega_{\gamma,\tau}^\gamma$, and all these tori are given by the family
 \begin{eqnarray*}
 \Psi_0: \mathds{T}^n\times \Omega \rightarrow B\times \mathds{T}^n, (\theta,\omega) \mapsto (0, \theta)
 \end{eqnarray*}
 of trivial embedding of $\mathds{T}^n$ over $\Omega$ into phase space.
 \end{remark}

\section{Example}\label{Applications}
When considering the dynamics of coorbital motion of two small moons orbiting a large planet with nearly circular orbits of almost equal radii, there are three small quantities to consider: (1). the ratio of the difference between the radii of the moons' orbits and the average radius of their orbits, denoted as $\varepsilon$, (2). the masses of the small moons, represented as $\varepsilon^a$, (3). the minimum angular separation of the moons, denoted as $\varepsilon^b$.

When $a\in (2, \frac{5}{2})$ and $b = a-2$, \cite{Cors1} proved the existence of coorbital motion. In detail,  J. Cors and G. Hall (\cite{Cors1}) got the following Hamiltonian expressed in variables $(\rho_1, \rho_2, \theta_1, \theta_2, R_1, R_2, \Phi_1, \Phi_2)$:
\begin{eqnarray}\label{noorder}
H = H_0 + \varepsilon^2 H_2 + \varepsilon^a H_a + \varepsilon^3 H_3 + \varepsilon^{a+1} H_{a+1} + \varepsilon^4 H_4 + O(\varepsilon^{a+2}),
\end{eqnarray}
where
\begin{eqnarray*}
H_0 &=& \frac{R_1^2}{2 \mu_1} + \frac{R_2^2}{2 \mu_2},\\
H_2 &=& \frac{1}{2} \big( \frac{\Phi_1^2}{\mu_1} + \frac{\Phi_2^2}{\mu_2} + \mu_1 \rho_1^2 + \mu_2 \rho_2^2 - 4 (\Phi_1 \rho_1 + \Phi_2 \rho_2)\big),\\
H_a &=& \frac{R_1^2}{2}+ \frac{R_2^2}{2} + (\mu_1 R_2 - \mu_2 R_1)\sin (\theta_2 - \theta_1) + (\mu_1 \mu_2 + R_1 R_2)\cos (\theta_2 - \theta_1)\\
&~&- \frac{\mu_1 \mu_2}{2} \csc (\frac{\theta_2 - \theta_1}{2}),\\
H_3 &=& -\frac{\Phi_1^2 \rho_1}{\mu_1} - \frac{\Phi_2^2 \rho_2}{\mu_2} - \mu_1 \rho_1^3 - \mu_2 \rho_2^3 + 3 (\Phi_1 \rho_1^2 + \Phi_2 \rho_2^2),\\
H_{a+1} &=& - \mu_1^2 \rho_1 - \mu_2^2 \rho_2 + \mu_1 \Phi_1+ \mu_2 \Phi_2+ \frac{1}{4} \mu_1 \mu_2 (\rho_1 + \rho_2)\csc ( \frac{\theta_2 - \theta_1}{2})\\
&~& (\mu_2 R_1 \rho_2 - \mu_1 R_2 \rho_1 - \Phi_2 R_1 + \Phi_1 R_2)\sin (\theta_2 - \theta_1)\\
&~& (\mu_2 \Phi_1 + \mu_1 \Phi_2 - \mu_1 \mu_2 (\rho_1 + \rho_2) )\cos (\theta_2 - \theta_1),\\
H_4 &=& \frac{3}{2} (\frac{\Phi_1^2 \rho_1^2}{\mu_1} + \frac{\Phi_2^2 \rho_2^2}{\mu_2} + \mu_1 \rho_1^4 + \mu_2 \rho_2^4) - 4 (\Phi_1 \rho_1^3 + \Phi_2 \rho_2^3).
\end{eqnarray*}
Here $\mu_1$ and $\mu_2$ are $O(1)$.  If $a\in (2, \frac{5}{2})$ and $b = a-2$, $(\ref{noorder})$ is a multiscale Hamiltonian with order relationship. And applying Han-Li-Yi Theorem in \cite{han}, \cite{Cors1} proved the existence of 4 dimension invariant tori. However, when $a> \frac{5}{2}$, this is a multiscale Hamiltonian without order relationship and Han-Li-Yi Theorem does not work. But using our result we can show the persistence of multiscale invariant tori.

Consider the following real-analytic  multi-scale Hamiltonian system:
\begin{eqnarray}
\label{EX1} H(I, \theta) &=& \tilde{H}(I)+ O(\varepsilon^{a+2} ),\\
\nonumber \tilde{H}(I) &=& H_0 + \varepsilon^2 H_2 + + \varepsilon^a H_a + \varepsilon^3 H_3 + \varepsilon^{a+1} H_{a+1} + \varepsilon^4 H_4,
\end{eqnarray}
where $ H_0(I) = I_1^2,$ $ H_2(I) =  I_2^2,$ $H_a(I) = I_3^2,$ $H_3(I) = I_4^2,$ $ H_{a+1} (I)= I_5^2,$ $ H_4(I) = I_6^2,$ $I= (I_1, I_2, I_3, I_4, I_5, I_6) \in D \subseteq \mathds{R}^6$, $D$ is a bounded close region, $O(\varepsilon^{a+2} )$ depend on $\theta$, and $a>2$. We have the following results:
\begin{cor}

\begin{enumerate}
  \item [$(a).$] There exists $\varepsilon_0>0$ and a family of Cantor sets  $D_\varepsilon \subset D$ for $0< \varepsilon \leq \varepsilon_0$, such that each $I_0\in D_\varepsilon$ corresponds to a real analytic, invariant, quasi-periodic $n-$torus $T_{I_0}^\varepsilon$, on which the frequency is $\partial_{I} \tilde{H}(I_0)$
  \item [$(b).$] Let $\Xi = \{I: \tilde{H}(I) = c\}$ be a given energy surface. Assume that $I_1^2 + \varepsilon^2 I_2^2 + I_3^2 + \varepsilon^3 I_4^2 + \varepsilon^{1+a} I_5^2 + \varepsilon^4 I_6^2 \neq 0$ on $\Xi$. Then, there exists $\varepsilon_0>0$ and a family of Cantor sets $\Xi_\varepsilon \subset \Xi$ for $0< \varepsilon \leq \varepsilon_0$, each $I_0 \in \Xi_\varepsilon$ corresponds to a real analytic, invariant, quasi-periodic $n-$torus, on which $\omega_\varepsilon = t\omega(I_0)$, where $t \rightarrow1$ as $\varepsilon\rightarrow 0$.
\end{enumerate}
\end{cor}

\begin{proof}
In fact, $O(\varepsilon^{a+2} )$ is a small perturbation and the frequency of the unperturbed system is $\omega = \frac{\partial \tilde{H}}{\partial I} = (2I_1,2\varepsilon^2 I_2, 2 \varepsilon^a I_3, 2\varepsilon^3 I_4, 2\varepsilon^{a+1} I_5, 2\varepsilon^4 I_6)$.
Assume the frequency is Diophantine, i.e., for fixed $\gamma>0$ and $\tau\geq n-1$, $|\langle k, \omega\rangle| \geq \frac{\tilde{\varepsilon}\gamma}{|k|^\tau},$ $0\neq k \in \mathds{Z}^n,$ where $\tilde{\varepsilon} = \min\{\varepsilon^2, \varepsilon^a, \varepsilon^3, \varepsilon^{a+1}, \varepsilon^4\}$. Since
\begin{eqnarray*}
&~&\det\frac{\partial^2 \tilde{H}}{ \partial I ^2}\\
&=& \det \left(
  \begin{array}{cccccc}
    2 & 0 & 0 & 0 & 0 & 0 \\
    0 & 2\varepsilon^2 & 0 & 0 & 0 & 0 \\
    0 & 0 & 2\varepsilon^a & 0 & 0 & 0 \\
    0 & 0 & 0 & 2\varepsilon^3 & 0 & 0 \\
    0 & 0 & 0 & 0 & 2\varepsilon^{a+1} & 0 \\
    0 & 0 & 0 & 0 & 0 & 2\varepsilon^4 \\
  \end{array}
\right)\\
 &=&64 \varepsilon^{10 + 2a},
\end{eqnarray*}
according to $(2)$ in Theorem $\ref{MainTheorem}$, we prove result $(a)$.

According to the assumption in $(b)$, we have
\begin{eqnarray*}
 &~&\left(
      \begin{array}{cc}
        \frac{\partial^2 \tilde{ H}}{\partial I^2} & \omega^T \\
        \omega & 0 \\
      \end{array}
    \right)\\
 &=&\left(
  \begin{array}{ccccccc}
    2 & 0 & 0 & 0 & 0 & 0&2I_1 \\
    0 & 2\varepsilon^2 & 0 & 0 & 0 & 0& 2\varepsilon^2 I_2\\
    0 & 0 & 2\varepsilon^a & 0 & 0 & 0& 2 \varepsilon^a I_3\\
    0 & 0 & 0 & 2\varepsilon^3 & 0 & 0& 2\varepsilon^3 I_4\\
    0 & 0 & 0 & 0 & 2\varepsilon^{a+1} & 0& 2\varepsilon^{a+1} I_5\\
    0 & 0 & 0 & 0 & 0 & 2\varepsilon^4& 2\varepsilon^4 I_6\\
    2I_1  &  2\varepsilon^2 I_2 & 2 \varepsilon^a I_3 & 2\varepsilon^3 I_4 & 2\varepsilon^{a+1} I_5 &  2\varepsilon^4 I_6 &0
     \end{array}
\right)\\
&=&1 2\varepsilon^2 2 \varepsilon^a 2 \varepsilon^3 2 \varepsilon^{a+1} 2 \varepsilon^4 (2 I_1)^2 - 2 2 \varepsilon^a 2 \varepsilon^3 2 \varepsilon^{a+1} 2 \varepsilon^4 (2 \varepsilon^2 I_2)^2\\
&~& - 2 \varepsilon^2 2 \varepsilon^3 2 \varepsilon^{a+1} 2\varepsilon^4 (2 \varepsilon^a I_3)^2 - 2 2\varepsilon^2 2 \varepsilon^a 2 \varepsilon ^{a+1} (2 \varepsilon^3 I_4)^2\\
&~& - 2 2\varepsilon^2 2 \varepsilon^a 2 \varepsilon^3 2 \varepsilon^4 (2 \varepsilon^{a+1} I_5)^2 - 2 2 \varepsilon^2 2 \varepsilon^a 2 \varepsilon^3 2 \varepsilon^{a+1} (2 \varepsilon^4 I_6)^2\\
&=&-128 \varepsilon^{10 + 2a} (I_1^2 + \varepsilon^2 I_2^2 + I_3^2 + \varepsilon^3 I_4^2 + \varepsilon^{1+a} I_5^2 + \varepsilon^4 I_6^2)\\
&\neq&0.
\end{eqnarray*}
Then, according to $(3)$ in Theorem $\ref{MainTheorem}$, we prove result $(b)$.\\
\end{proof}

\noindent\textbf{Acknowledgements}
The first author was supported by NSFC (12201243, 12371191), China Postdoctoral Science Foundation (2021M701396, 2022T150262) and ERC Grant(885707). The second author was supported by NSFC (12071175, 12471183). The third author was supported by NSFC (12371191, 12071175).

\baselineskip 9pt \renewcommand{\baselinestretch}{1.08}

\begin{appendices}\section{Multi-scale matrix}\label{Inverse}
The following Lemma comes from paper $\cite{Horn}$.
\begin{lemma}\label{LMU}
Let $A$, $E$ $\in$ $M_n$ and suppose that $A$ is normal. If $\hat{\lambda}$ is an eigenvalue of $A+E$, then there is an eigenvalue $\lambda$ of $A$ such that $|\hat{\lambda} - \lambda| \leq |||E|||_2$, where $|||E|||_2 = (tr E^* E)^{\frac{1}{2}}$.
\end{lemma}

\begin{lemma}\label{Eigenvalue}
Let $A= \varepsilon_0 A_0 + \cdots + \varepsilon_m A_m$ $\in$ $M_{n,n},$ where $0< \varepsilon_k \ll1$, $A_k = (a_{ij}^k)_{n\times n}$, $0\leq k\leq m$ . Denote the eigenvalue of $AA^*$ by $\lambda_1\leq \cdots\leq \lambda_n$. Then $\lambda_1 > c \min\limits_{1\leq j \leq n} \{\varepsilon_j\}$, where $c$ is positive and depends on $a_{ij}^k$, $1\leq i,j \leq n$, $0\leq k\leq m$.
\end{lemma}

\begin{proof}
Denote $A^*$ the conjugate transpose of $A$.
Directly, we have
\begin{eqnarray*}
AA^* &=& (\varepsilon_0 A_0 + \cdots + \varepsilon_m A_m) (\varepsilon_0 A_0^* + \cdots + \varepsilon_m A_m^*) \\
&=& \sum\limits_{j_1=0}^m \varepsilon_{j_1} A_{j_1}\sum\limits_{j_2 =0}^m \varepsilon_{j_2} A_{j_2}^*\\
&=&\sum\limits_{j_1=0}^m \sum\limits_{j_2=0}^m \varepsilon_{j_1}\varepsilon_{j_2} A_{j_1} A_{j_2}^*\\
&=& (\sum\limits_{k=1}^n \sum\limits_{j_1=0}^m \sum\limits_{j_2=0}^m \varepsilon_{j_1}\varepsilon_{j_2} a_{ik}^{j_1} a_{kj}^{j_2})_{n\times n}.
\end{eqnarray*}
Let $\hat{A} = \big(\max\limits_{1\leq j\leq m} \{\varepsilon_j\}\big)^2 I_n$ and $P_{\varepsilon} = AA^*- \hat{A}.$ Denote the eigenvalue of $AA^*$ by $\breve{\lambda}_1\leq \cdots\leq \breve{\lambda}_n.$ Obviously, $\hat{A}^* \hat{A} = \hat{A} \hat{A}^*$, i.e., $\hat{A}$ is normal.  According to Lemma \ref{LMU}, we have
\begin{eqnarray}\label{EQM1}
|\breve{\lambda}_1 -  \max\limits_{1\leq j\leq m} \{\varepsilon_j^2\}| \leq |||P_{\varepsilon}|||_2,
\end{eqnarray}
where $|||P_{\varepsilon}|||_2 = (tr P_{\varepsilon}^* P_{\varepsilon} )^{\frac{1}{2}}$. Denote $b_{ij}=\sum\limits_{k=1}^n \sum\limits_{j_1=0}^m \sum\limits_{j_2=0}^m \varepsilon_{j_1}\varepsilon_{j_2} a_{ik}^{j_1} a_{kj}^{j_2}$. Then $AA^* = (b_{ij})_{n\times n}.$ Let $P_{\varepsilon} = (c_{ij})_{n\times n}$, where $c_{ij}= b_{ij}$, $i\neq j$, $c_{ij} = b_{ij} - \big(\max\limits_{1\leq j\leq m} \{\varepsilon_j\}\big)^2$. Thus
\begin{eqnarray*}
|||P_{\varepsilon}|||_2 &=& (tr P_{\varepsilon}^* P_{\varepsilon} )^{\frac{1}{2}}\\
&=& ( \sum\limits_{j=1}^n\sum\limits_{i=1}^n \bar{c}_{ij} c_{ij} )^{\frac{1}{2}},
\end{eqnarray*}
which implies $\min\limits_{j}\{\varepsilon_j^2\}\leq |||P_{\varepsilon}|||_2 \leq \max\limits_{j}\{\varepsilon_j^2\}.$ Hence,
\begin{eqnarray*}
\breve{\lambda}_1 &\geq& \max\limits_{j}\{\varepsilon_j^2\} -|||P_\varepsilon|||_2\\
&\geq& \max\limits_{j}\{\varepsilon_j^2\} - ( \sum\limits_{j=1}^n\sum\limits_{i=1}^n \bar{c}_{ij} c_{ij} )^{\frac{1}{2}}\\
&\geq& c \min\limits_{j}\{\varepsilon_j^2\},
\end{eqnarray*}
where $c$ is positive and depends on $a_{ij}^k$, $1\leq i,j \leq n$, $0\leq k\leq m$.
\end{proof}
\end{appendices}
\end{document}